\documentclass[leqno]{amsart}
\usepackage{amscd,amsmath,amsopn,amssymb,amsthm}
\usepackage[english]{babel}
\usepackage{hyperref}
\DeclareMathOperator{\diag}{diag}
\DeclareMathOperator{\Ad}{Ad}
\DeclareMathOperator{\Lin}{Lin}
\DeclareMathOperator{\trace}{trace}
\DeclareMathOperator{\const}{const}
\DeclareMathOperator{\inj}{inj}
\DeclareMathOperator{\can}{can}
\DeclareMathOperator{\Isom}{Isom}

\renewenvironment{proof}[1][Proof]{\textbf{#1.} }
{\ \rule{0.5em}{0.5em}}

\newtheorem{theorem}{Theorem}
\newtheorem{pred}{Proposition}
\newtheorem{lemma}{Lemma}
\newtheorem{corollary}{Corollary}
\newtheorem{definition}{Definition}
\newtheorem{remark}{Remark}
\newtheorem{example}{Example}
\newtheorem{question}{Question}

\begin{document}

\title[Clifford-Wolf homogeneous \dots]{Clifford-Wolf homogeneous Riemannian manifolds}

\author{V.N.~Berestovski\u\i , Yu.G.~Nikonorov}

\address{Berestovskii Valerii Nikolaevich \newline
Omsk Branch of Sobolev Institute of Mathematics SD RAS
\newline 644099, Omsk, ul. Pevtsova, 13, Russia}

\email{berestov@ofim.oscsbras.ru}

\address{Nikonorov Yurii  Gennadievich \newline
Rubtsovsk Industrial Institute
\newline of Altai State Technical University after
I.I.~Polzunov
\newline 658207, Rubtsovsk, Traktornaya, 2/6, Russia}

\email{nik@inst.rubtsovsk.ru}

\begin{abstract}
In this paper, using connections between Clifford-Wolf isometries
and Killing vector fields of constant length on a given Riemannian
manifold, we classify simply connected Clifford-Wolf homogeneous
Riemannian manifolds. We also get the classification of complete
simply connected Riemannian manifolds with the Killing property
defined and studied previously by J.E.~D'Atri and H.K.~Nickerson. In
the last part of the paper we study properties of Clifford-Killing
spaces, that is, real vector spaces of Killing vector fields of
constant length, on odd-dimensional round spheres, and discuss
numerous connections between these spaces and various classical
objects.

\vspace{2mm}

\vspace{2mm} \noindent Key words and phrases: homogeneous space,
Killing field of constant length, isometry group of Riemannian
manifold, Clifford-Wolf translation, Clifford-Wolf homogeneous
Riemannian manifold, $\delta$-homogeneous Riemannian manifold,
Riemannian manifold with the Killing property, Radon-Hurwitz
function, Clifford-Killing space.
\end{abstract}

\maketitle

\tableofcontents

\section*{Introduction and the main result}

In this paper the authors give complete isometric classification of
simply connected manifolds in the title. Recall that a Clifford-Wolf
translation of a metric space is an isometry of the space onto
itself moving all its points one and the same distance \cite{F},
\cite{Wolf2}. A metric space is Clifford-Wolf homogeneous (or
CW-homogeneous), if for any two its points there exists a
Clifford-Wolf translation of the space moving one of these points to
another \cite{BerP}.

It is not difficult to see that Euclidean spaces, odd-dimensional round spheres,
and Lie groups with
bi-invariant Riemannian metrics, as well as direct metric products of
Clifford-Wolf homogeneous Riemannian spaces are Clifford-Wolf
homogeneous. The main result of this paper states that in simply
connected case the opposite statement is also true. More exactly,

\begin{theorem}\label{main}
A simply connected (connected) Riemannian manifold is Clifford-Wolf
homogeneous if and only if it is a direct metric
product of an Euclidean space, odd-dimensional spheres of constant
curvature and simply connected compact simple Lie groups supplied
with bi-invariant Riemannian metrics (some of these factors could be
absent).
\end{theorem}

In the first section we give in more details necessary definitions
of spaces under investigation, discuss examples of Clifford-Wolf
translations and groups consisting of them.

Clifford-Wolf homogeneous space is a special case of so-called
$\delta$-homogeneous metric space (see definition in Section 1 or in
\cite{BerP}). On the ground of the author's structure results on
$\delta$-homogeneous Riemannian manifolds in \cite{BerNikDGA}, we
present in the second section analogous results for CW-homogeneous
manifolds.

In the third sections, the authors elaborate main technical means
and tools for the further investigation, namely Killing vector
fields of constant length, their properties, and especially the very
top result of Theorem 5 about the nullity of covariant derivative of
the curvature tensor on Riemannian manifold when one uses three
times one and the same Killing vector field of constant length.

The subject of the fourth section is well presented by its title. In
particular, we find a connection of the notion of more general
restrictively CW-homogeneous manifolds with the possibility of
presenting geodesics as integral curves of Killing vector fields of
constant length. With the help of author's investigations of Killing
vector fields on symmetric spaces in \cite{BerNikTrans}, it is
proved that for symmetric spaces this more general notion is
equivalent to the old one.

In the next section the main result is proved. First of all, on the
ground of mentioned Theorem 5 and one result in \cite{Bes}, we
immediately get that any simply connected CW-homogeneous Riemannian
manifold is symmetric. After this, thoroughly, even not so long,
study of the CW-homogeneity condition in symmetric case permits to
finish proof.

In Section 6 the Clifford-Killing spaces, that is, real vector
spaces of Killing vector fields of constant length, are introduced.

Manifolds in the title of the seventh section have been defined and
investigated by D'Atri and Nickerson in the paper \cite{DatN}. They
are exactly Riemannian manifolds which locally admit
Clifford-Killing spaces of dimension equal to the dimension of the
manifold. We proved in Theorem 11 that in simply connected complete
case, these manifolds are classified similar to CW-homogeneous
spaces in Theorem \ref{main}, but among odd-dimensional spheres one
should leave only seven-dimensional one.

It is quite interesting that classical famous results of Radon and
Hurwitz in \cite{Rad} and \cite{H}, discussed in Section 8, can be
interpreted exactly as a construction of Clifford-Killing spaces of
maximal dimension on round odd-dimensional spheres, and famous
Radon-Hurwitz function gives their dimensions.

In the ninth section, a close connection of Clifford-Killing spaces
on round odd-dimensional spheres with Clifford algebras and modules
is established. This permits to classify all Clifford-Killing vector
spaces on round spheres in Theorems 16 (and 17) up to
self-isometries of spheres (preserving the orientation).

Radon in his paper \cite{Rad} observed that his considerations of
the Hurwitz Question 2 in Section 8 are closely connected with some
(topological) spheres in the Lie group $O(2n)$, $n\geq 2$. We proved
that any this Radon's sphere is a totaly geodesic sphere in $O(2n)$
supplied with a bi-invariant Riemannian metric $\mu$.

In the next section we deal with more general question of totally
geodesic spheres in $(SO(2n), \mu)$ related to triple Lie systems in
the Lie algebra of $SO(2n)$ and Clifford-Killing spaces on round
spheres $S^{2n-1}$. In particular, it is proved, besides results,
similar to results in previous section, that Radon's spheres
coincide with totally geodesic Helgason's spheres (of constant
sectional curvature $k$ equal to maximal sectional curvature of
$(SO(2n),\mu)$) from the paper \cite{HelP} if and only if $n=2$. In
Proposition 12, the curvature $k$ is calculated.

In the last section, Lie algebras in Clifford-Killing
spaces on round odd-dimensional spheres are studied.

\medskip
{\bf Acknowledgements.}
We thank P.~Eberlein and J.A.~Wolf for helpful discussion of this project.
The project is supported by
the State Maintenance Program for the Leading Scientific Schools
of the Russian Federation (grants NSH-5682.2008.1).
The first author is supported in part by RFBR (grant 08-01-00067-a).

\section{Preliminaries}

In the course of the paper, if the opposite is not stated, a
Riemannian manifold means a connected $C^{\infty}$-smooth Riemannian
manifold; the smoothness of any object means
$C^{\infty}$-smoothness. For a Riemannian manifold $M$ and its point
$x\in M$ by $M_x$ we denote the tangent (Euclidean) space to $M$ at
the point $x$. For a given Riemannian manifold $(M,g)$ by $\rho$ we
denote the inner (length) metric generated by the Riemannian metric
tensor $g$ on $M$.

Recall that a {\it Clifford-Wolf translation} in $(M,g)$ is an
isometry $s$ moving all points in $M$ one and the same distance,
i.~e. $\rho(x,s(x)) \equiv \const$ for all $x\in M$. Notice that
Clifford-Wolf translations are often called \textit{Clifford
translations} (see for example \cite{W} or \cite{KN}), but we follow
in this case the terminology of the paper \cite{F}. Clifford-Wolf
translations naturally appear in the investigation of homogeneous
Riemannian coverings of homogeneous Riemannian manifolds
\cite{W,KN}. H.~Freudenthal classified in the paper \cite{F} all
individual Clifford-Wolf translations on symmetric spaces. Let us
indicate yet another construction of such transformations.
Suppose
that some isometry group $G$ acts transitively on a Rimannian
manifold $M$ and $s$ is any element of the centralizer of $G$ in the full isometry group $\Isom(M)$ of $M$.
Then $s$ is a Clifford-Wolf translation (in particular, if the center $Z$ of the group $G$ is not discrete,
then every one-parameter subgroup in $Z$ consists of Clifford-Wolf
translations on $M$). Indeed, if $x$ and $y$ are some points of
manifold $M$, then there is $g\in G$ such that $g(x)=y$. Thus
$$
\rho(x,s(x))=\rho(g(x),g(s(x)))=\rho(g(x),s(g(x)))=\rho(y,s(y)).
$$
For symmetric spaces, this result can be inverted
due to V.~Ozols (see \cite{Ozols4}, Corollary 2.7):
If $M$ is symmetric, then for any Clifford-Wolf
translation $s$ on $M$, the centralizer of $s$ in $\Isom(M)$
acts transitively on $M$.

 Notice that several classical Riemannian
manifolds possesses a one-parameter group of Clifford-Wolf
translations.
For instance, one knows that among irreducible compact simply
connected symmetric spaces only odd-dimensional spheres, the spaces
$SU(2m)/Sp(m)$, $m\geq 2$, and simple compact Lie groups, supplied
with some bi-invariant Riemannian metric, admit one-parameter groups
of Clifford-Wolf translations \cite{Wolf62}.

\begin{definition}[\cite{BerP}]\label{cw}
An inner metric space $(M,\rho)$ is called {\it Clifford-Wolf
homogeneous} if for every  two points $y,z$ in $M$ there exists a
Clifford-Wolf translation of the space $(M,\rho)$ moving $y$ to $z$.
\end{definition}

Let us consider some examples. Obviously, every Euclidean space
$\mathbb{E}^n$ is Clifford-Wolf homogeneous. Since $\mathbb{E}^n$ can be treated as a
(commutative) additive vector group with a bi-invariant inner
product, the following example can be considered as a
generalization.

\begin{example}\label{ecwh1}
Let $G$ be a Lie group supplied with a bi-invariant Riemannian
metric $\rho$. In this case both the group of left shifts $L(G)$ and
the group of right shifts $R(G)$ consist of Clifford-Wolf isometries
of $(G,\rho)$. Therefore, $(G,\rho)$ is Clifford-Wolf homogeneous.
Note also that in \cite{BerD} the following result has been proved:
A Riemannian manifold $(M,g)$ admits a transitive group $\Gamma$ of
Clifford-Wolf translations if and only if it is isometric to some
Lie group $G$ supplied with a bi-invariant Riemannian metric.
\end{example}

\begin{example}\label{ecwh2}
Every odd-dimensional round sphere $S^{2n-1}$ is Clifford-Wolf
homogeneous. Indeed, $S^{2n-1}=\{\xi=(z_1,\dots, z_n)\in
\mathbb{C}^n: \sum_{k=1}^{n}|z_k|^2=1\}$. Then the formula
$\gamma(t)(\xi)=e^{it}\xi$ defines a one-parameter group of
Clifford-Wolf translations on $S^{2n-1}$ with all orbits as geodesic
circles. Now, since $S^{2n-1}$ is homogeneous and isotropic, any its
geodesic circle is an orbit of a one-parameter group of
Clifford-Wolf translations, and so $S^{2n-1}$ is Clifford-Wolf
homogeneous. Note that $S^1$ and $S^3$ can be treated as the Lie groups
$SO(2)$ and $SU(2)$ with bi-invariant Riemannian metrics.
\end{example}

Note also that any direct metric product of Clifford-Wold
homogeneous Riemannian manifolds is Clifford-Wold homogeneous
itself. On the other hand, the condition for a Riemannian manifold
to be Clifford-Wolf homogeneous, is quite strong. Therefore, one
should hope to get a complete classification of such manifolds.

\begin{definition}\label{rcw}
An inner metric space $(M,\rho)$ is called {\it restrictively
Clifford-Wolf homogeneous} if for any $x\in M$ there exists a number
$r(x)>0$ such that for any two points $y,z$ in open ball $U(x,r(x))$
there exists a Clifford-Wolf translation of the space $(M,\rho)$
moving $y$ to $z$.
\end{definition}

The notion of (restrictively) Clifford-Wolf homogeneous metric space
is related to the notion of $\delta$-homogeneous metric space.
Recall the following definition.

\begin{definition}[\cite{BerP}]\label{o}
Let $(X,d)$ be a metric space and $x\in X$. An isometry $f:X
\rightarrow X$ is called {\it a $\delta(x)$-translation}, if $x$ is
a point of maximal displacement of $f$, i.~e. for every $y \in X$ the
relation $d(y,f(y))\leq d(x,f(x))$ holds. A metric space $(X,d)$ is
called {\it $\delta$-homogeneous}, if for every $x,y \in X$ there
exists a $\delta(x)$-translation of $(X,d)$, moving $x$ to $y$.
\end{definition}

It is easy to see that every restrictively Clifford-Wolf homogeneous
locally compact complete inner metric space is $\delta$-homogeneous
(see Proposition 1 in \cite{BerNikDGA}). Therefore, we can apply all
results on $\delta$-homogeneous Riemannian manifolds to
(restrictively) Clifford-Wolf homogeneous Riemannian manifolds
\cite{BerNikDGA}.

We shall need also the following

\begin{definition}\label{SKWH}
An inner metric space $(X,d)$ is called strongly Clifford-Wolf
homogeneous if for every two points $x,y\in X$ there is a
one-parameter group $\gamma(t)$, $t\in \mathbb{R}$, of Clifford-Wolf
translations of the space $(X,d)$ such that for sufficiently small
$|t|$, $\gamma(t)$ shifts all points of $(X,d)$ to distance $|t|$,
and $\gamma(s)(x)=y$, where $d(x,y)=s$.
\end{definition}

It is clear that any strongly Clifford-Wolf homogeneous inner metric
space is a Clifford-Wolf homogeneous and Clifford-Wolf homogeneous
inner metric space is a restrictively Clifford-Wolf homogeneous. It
is a natural question: are the above three classes pairwise distinct
(in particular, in the case of Riemannian manifolds)?

\section{Some structure results}

\begin{lemma}\label{first}
Every restrictively Clifford-Wolf homogeneous Riemannian manifold is homogeneous and, consequently,
complete.
\end{lemma}

\begin{proof}
Let $(M,g)$ be a restrictively Clifford-Wolf homogeneous Riemannian manifold
and let $\Isom$ be its full isometry group. By
Definition \ref{rcw}, the orbit $O$ of a point $x\in M$ under the action of $\Isom$
is open subset of $M$. Since $O$ is clearly closed in $M$ and $M$ assumed to be connected,
then $O=M$.
\end{proof}

\begin{theorem}\label{product}
Every (restrictively) Clifford-Wolf homogeneous Riemannian manifolds
is $\delta$-homogeneous, has non-negative sectional curvature, and
is a direct metric product of Euclidean space and a compact
(restrictively) Clifford-Wolf homogeneous Riemannian manifold.
\end{theorem}

\begin{proof}
The first assertion follows from Proposition 1 in \cite{BerNikDGA} and Lemma \ref{first}.
The second one follows from the fact that every $\delta$-homogeneous
Riemannian manifold has nonnegative sectional curvature (see
\cite{BerNikDGA} and \cite{BerP} for some more general results). The
third assertion easily follows from Toponogov's theorem in \cite{T},
stating that every complete Riemannian manifold $(M,\mu)$ with
nonnegative sectional curvature, containing a metric line, is
isometric to a direct Riemannian product $(N,\nu)\times \mathbb{R}$
(more general results in this direction could be found in
\cite{BerP}) and the fact that any Clifford-Wolf translation
preserves a product structure on a metric product of two Riemanian
manifolds, and so is a direct product of suitable Clifford-Wolf
translations on these manifolds (see Theorem 3.1.2 in
\cite{Wolf62}).
\end{proof}

\begin{theorem}[Corollary 3 in \cite{BerNikDGA}]\label{cover}
Any Riemanniang covering of a (restrictively) Clifford-Wolf
homogeneous Riemannian manifold, is (restrictively) Clifford-Wolf
homogeneous itself.
\end{theorem}

\begin{theorem}\label{product2}
Let $M$ be a simply connected (restrictively) Clifford-Wolf
homogeneous Rieman\-ni\-an manifold and $M=M_{0}\times M_{1}\times
\dots \times M_{k}$ its de~Rham decomposition, where $M_0$ is
Euclidaen space, and the others $M_i$, $1\leq i \leq k$, are simply
connected compact Riemannian manifolds. Then every $M_i$ is a
(restrictively) Clifford-Wolf homogeneous Riemannian manifolds.
Moreover, any isometry $f$ of $M$ is a Clifford-Wolf translation if
and only if it is a product of some Clifford-Wolf translations $f_i$
on $M_i$.
\end{theorem}

\begin{proof}
Obviously, the first assertion is a consequence of the second one,
that was proved in Corollary 3.1.3 in \cite{Wolf62}.
\end{proof}

\section{On Killing vector fields of constant length}

Here we consider some properties of Killing vector fields of
constant length on Riemanniann manifolds. Recall that a vector field
$X$ on a Riemannian manifold $(M,g)$ is called {\it Killing} if $L_X
g=0$. For a Killing vector field $X$ it is useful to consider the
operator $A_X$ defined on vector fields by the formula $A_X
V=-\nabla_VX$. It is clear that $A_X=L_X-\nabla_X$.

All assertions of the following lemma are well known (see e.g.
\cite{KN}).

\begin{lemma}\label{Kost}
Let $X$ be a Killing vector field on the Riemannian manifold
$(M,g)$. Then the following statements hold:

1) For any vector fields $U$ and $V$ on $M$ holds the equality
$g(\nabla_U X,V)+g(U,\nabla_V X)=0$. In other words, the operator
$A_X$ is skew-symmetric.

2) For every vector field $U$ on $M$ holds the equality
$$
R(X,U)=[\nabla_X,\nabla_U]-\nabla_{[X,U]}=[\nabla_U,A_X],
$$
where $R$ is the curvature tensor of $(M,g)$.

3) For any Killing vector field $X$ and for any vector fields
$U,V,W$ on a Riemannian manifold $(M,g)$ the following formula
holds:
$$
-g(R(X,U)V,W)=g(\nabla_U \nabla_V X,W)+g(\nabla_UV, \nabla_W X),
$$
where $R$ is the curvature tensor of $(M,g)$.
\end{lemma}

\begin{proof}
It is clear that $X\cdot g(U,V)=g(\nabla_X U,V)+g(U,\nabla_X V)$ and
$X\cdot g(U,V)=g([X,U],V)+g(U,[X, V])$. Therefore,
$$
g(\nabla_U X,V)+g(U,\nabla_V X)=g(\nabla_X U-[X,U],V)+g(U,\nabla_X
V-[X, V])=0,
$$
that proves the first assertion.

The second assertion is proved in Lemma 2.2 of \cite{Kostant55} (see
also Proposition 2.2 of Chapter 6 in \cite{KN}). The third assertion
follows from the previous one:
$$
g(R(X,U)V,W)=g(\nabla_U A_X V,W)- g(A_X \nabla_U V,W)=
$$
$$
-g(\nabla_U \nabla_V X,W)+ g(\nabla_U V,A_X W)=-g(\nabla_U \nabla_V
X,W)- g(\nabla_U V, \nabla_W X),
$$
because the operator $A_X$ is skew symmetric for a Killing vector
field $X$.
\end{proof}

Now we recall some well known properties of Killing vector field of
constant length.

\begin{lemma}\label{tr1}
Let $X$ be a Killing vector field $X$ on a Riemannian manifold
$(M,g)$. Then the following conditions are equivalent:

1) $X$ has constant length on $M$;

2) $\nabla_XX=0$ on $M$;

3) every integral curve of the field $X$ is a geodesic in $(M,g)$.
\end{lemma}

\begin{proof}
It is suffices to note that the length of $X$ is constant along any
its integral curve, and for any Killing vector field $X$ and
arbitrary smooth vector field $Y$ on $(M,g)$ we have the following
equality:
$$
0=(L_Xg)(X,Y)=X\cdot g(X,Y)-g([X,X],Y)-g(X,[X,Y])= g(\nabla_X
X,Y)+g(X,\nabla_X Y)-
$$
$$
-g(X,[X,Y])= g(\nabla_X X,Y)+g(X,\nabla_Y X)= g(\nabla_X
X,Y)+\frac{1}{2}Y\cdot g(X,X).
$$
\end{proof}

\begin{pred}\label{criv2}
Let $Z$ be a Killing vector field of constant length and $X,Y$
arbitrary vector fields on a Riemannian manifold $(M,g)$. Then the
formula
$$
g(\nabla_X Z,  \nabla_Y Z)= g(R(X,Z)Z,Y)=g(R(Z,Y)X,Z)
$$
holds on $M$.
\end{pred}

\begin{proof}
Since $g(Z,Z)=\const$, then $X\cdot g(Z,Z)=2g(\nabla_X Z,Z)=0$.
Therefore,
$$
0=Y\cdot g(\nabla_X Z,Z)= g(\nabla_Y\nabla_X Z,Z)+ g(\nabla_X
Z,\nabla_Y Z).
$$
By Assertion 3) in Lemma \ref{Kost} we get
$$
g(\nabla_Y\nabla_X
Z,Z)=-g(R(Z,Y)X,Z)-g(\nabla_YX,\nabla_ZZ)=-g(R(Z,Y)X,Z),
$$
because $\nabla_ZZ=0$. Therefore, $g(\nabla_X Z,  \nabla_Y Z)=
g(R(Z,Y)X,Z)$. The formula $g(\nabla_X Z,  \nabla_Y Z)=
g(R(X,Z)Z,Y)$ follows from symmetries of the curvature tensor.
\end{proof}

\begin{lemma}\label{symm0}
For every Killing vector field of constant length $Z$ and any vector
fields $X,Y$ on a Riemannian manifold $(M,g)$ the following
equalities hold:
$$
g(R(X,Z)Z,\nabla_Y Z)+g(R(Y,Z)Z,\nabla_X Z)=0,
$$
$$
g(\nabla_Z \nabla_Y Z, \nabla_X Z)=g(R(X,Z)Z, \nabla_Z Y).
$$
\end{lemma}

\begin{proof}
Let us prove the first equality. By Proposition \ref{criv2} we have
$$
g(R(X,Z)Z,\nabla_Y Z)+g(R(Y,Z)Z,\nabla_X Z)=
$$
$$
g(\nabla_X Z,\nabla_{\nabla_YZ} Z)+g(\nabla_Y Z,\nabla_{\nabla_XZ}
Z)= g(U,\nabla_V Z)+g(V,\nabla_U Z),
$$
where $U=\nabla_XZ$ and $V=\nabla_YZ$. Now, using Assertion 1) in
Lemma \ref{Kost}, we get $g(U,\nabla_V Z)+g(V,\nabla_U Z)=0$.

Further, since $\nabla_Z Z=0$ and for any vector field $W$ the
equalities
$$
\nabla_Z \nabla_W=\nabla_W \nabla_Z+\nabla_{[Z,W]}+R(Z,W), \quad
\nabla_Z W=\nabla_W Z+[Z,W]
$$
hold, we get (using Proposition \ref{criv2})
$$
g(\nabla_Z \nabla_Y Z, \nabla_X Z)=g(\nabla_{[Z,Y]}Z,\nabla_X
Z)+g(R(Z,Y)Z,\nabla_X Z)=
$$
$$
g(R(X,Z)Z,[Z,Y])+g(R(Z,Y)Z,\nabla_X Z)=
$$
$$
g(R(X,Z)Z,\nabla_Z Y)-g(R(X,Z)Z,\nabla_Y Z)-g(R(Y,Z)Z,\nabla_X Z).
$$
On the other hand, we have proved that $g(R(X,Z)Z,\nabla_Y
Z)+g(R(Y,Z)Z,\nabla_X Z)=0$, hence we get the second equality.
\end{proof}

Now we can prove the following theorem, that plays a key role in our
study.

\begin{theorem}\label{symm1}
For any Killing vector field of constant length $Z$ on a Riemannian
manifold $(M,g),$
$$ (\nabla_Z R)(\cdot, Z)Z \equiv 0.$$
\end{theorem}

\begin{proof}
It suffices to prove that $g((\nabla_Z R)(X,Z)Z,Y)=0$ for every
vector fields $X$ and $Y$ on $M$. From Proposition \ref{criv2} we
know that
$$
g(\nabla_X Z,  \nabla_Y Z)=g(R(X,Z)Z,Y).
$$
Therefore,
$$
Z\cdot g(\nabla_X Z,  \nabla_Y Z)= Z\cdot g(R(X,Z)Z,Y).
$$
Further, by Lemma \ref{symm0}, we get
$$
Z\cdot g(\nabla_X Z,  \nabla_Y Z)=g(\nabla_Z \nabla_X Z, \nabla_Y
Z)+g(\nabla_X Z, \nabla_Z \nabla_Y Z)=
$$
$$
g(R(Y,Z)Z, \nabla_Z X)+g(R(X,Z)Z, \nabla_Z Y).
$$
On the other hand,
$$
Z\cdot g(R(X,Z)Z,Y)=g((\nabla_Z R)(X,Z)Z,Y)+g(R(\nabla_Z
X,Z)Z,Y)+g(R(X,Z)Z,\nabla_Z Y)=
$$
$$
g((\nabla_Z R)(X,Z)Z,Y)+g(R(Y,Z)Z,\nabla_ZX)+g(R(X,Z)Z,\nabla_Z Y).
$$
Combining the equations above, we get $g((\nabla_Z R)(X,Z)Z,Y)=0$,
that proves the theorem.
\end{proof}

Note, that the condition $(\nabla_Z R)(\cdot, Z)Z=0$ means that for
any geodesic $\gamma$ that is an integral curve of the field $Z$, a
derivative of any normal Jacobi field along $\gamma$ is also a
normal Jacobi field (see Section 2.33 in \cite{Bes1}).

\section{Interrelations of Clifford-Wolf isometries and Killing fields with constant length}

There exists a connection between Killing vector fields of constant
length and Clifford-Wolf translations in a Riemannian manifold
$(M,g)$. The following proposition is evident.

\begin{pred}\label{tr3}
Suppose that a one-parameter isometry group $\gamma(t)$ on $(M,g)$, generated
by a Killing vector field $X$, consists of Clifford-Wolf
translations. Then $X$ has constant length.
\end{pred}

Proposition \ref{tr3} can be partially inverted. More exactly, we
have

\begin{pred}[\cite{BerNikTrans}]\label{tr4}
Suppose a Riemannian manifold $(M,g)$ has the injectivity radius,
bounded from below by some positive constant (in particularly, this
condition is satisfied for arbitrary compact or homogeneous
manifold), and $X$ is a Killing vector field on $(M,g)$ of constant
length. Then isometries $\gamma(t)$ from the one-parameter motion
group, generated by the vector field $X$, are Clifford-Wolf
translations if $t$ is close enough to $0$.
\end{pred}

\begin{theorem}[\cite{BerNikDGA}]\label{killi}
Let $(M,g)$ be a compact homogeneous Riemannian manifold. Then there
exists a positive number $s>0$ such that for arbitrary motion $f$ of
the space $(M,g)$ with maximal displacement $\delta$, which is less
than $s$, there is a unique Killing vector field $X$ on $(M,g)$ such
that $\max_{x\in M}\sqrt{g(X(x),X(x))}=1$ and
$\gamma_{X}(\delta)=f$, where $\gamma_{X}(t)$, $t\in \mathbb{R}$, is
the one-parameter motion group of $(M,\mu)$, generated by the field
$X$. If moreover $f$ is a Clifford-Wolf translation, then the
Killing vector field $X$ has unit length on $(M,g)$.
\end{theorem}

From previous results we easily get the following

\begin{theorem}\label{rcwhom}
A Riemannian manifold $(M,g)$ is restrictively Clifford-Wolf
homogeneous if and only if it is complete and every geodesic $\gamma$ in $(M,g)$ is an
integral curve of a Killing vector field of constant length on
$(M,g)$.
\end{theorem}

\begin{proof}
Since by Theorem \ref{product} every restrictively Clifford-Wolf
homogeneous Riemannian manifold is a direct metric product of
Euclidean space and a compact restrictively Clifford-Wolf
homogeneous Riemannian manifold, it is sufficient to consider the compact
case.

Suppose that $(M,g)$ is restrictively Clifford-Wolf homogeneous. Let
us take any $x \in M$ and any geodesic $\gamma$ through the point
$x$. Consider the number $r(x)>0$ as in Definition \ref{rcw}, and
take any point $y\in \gamma$, $y\neq x$, such that
$\varepsilon:=\rho(x,y) < \min\{r(x),r_{\inj}(x),s\}$, where
$r_{\inj}(x)$ is the injectivity radius at $x$ and $s$ is taken from
the statement of Theorem \ref{killi}. Then by Definition \ref{rcw}
there is a Clifford-Wolf translation moving $x$ to $y$. By Theorem
\ref{killi}, there is a unit Killing field $X$ on $(M,g)$ such that
$\gamma_{X}(\varepsilon)=f$, where $\gamma_{X}(t)$, $t\in
\mathbb{R}$ is the one-parameter motion group generated by the field
$X$. By Lemma \ref{tr1} an integral curve of $X$ through $x$ is a
geodesic, that evidently coincides with $\gamma$. Since $x$ is an
arbitrary point of $M$, then any geodesic $\gamma$ on $(M,g)$ is an
integral curve of a Killing field of constant length on $(M,g)$.

Now, suppose that every geodesic $\gamma$ on $(M,g)$ is an integral
curve of a Killing field of constant length on $(M,g)$. By
Proposition \ref{tr4} there is $\delta>0$ such that for any unit
Killing vector field $X$ on $(M,g)$, all isometries $\gamma(t)$,
$|t|<\delta$, from the one-parameter motion group, generated by $X$,
are Clifford-Wolf translations. Now, for any $x,y \in M$ such that
$\rho(x,y) <\delta$, there is a geodesic $\gamma$ through $x$ and
$y$ such that the segment between $x$ and $y$ has length
$\rho(x,y)$. Let $X$ be a unit Killing field such that $\gamma$ is
an integral curve of $X$. Consider a one-parameter motion group
$\gamma(t)$, $t\in \mathbb{R}$, generated by $X$. Then
$s_1=\gamma(\rho(x,y))$ and $s_2=\gamma(-\rho(x,y))$ are
Clifford-Wolf isometries on $(M,g)$, and either $s_1(x)=y$ or
$s_2(x)=y$. Therefore, $(M,g)$ is restrictively Clifford-Wolf
homogeneous.
\end{proof}

Let us cite the following result.

\begin{theorem}[\cite{BerNikTrans}]\label{Posit}
Let $M$ be a symmetric Riemannian space, $X$ is a Killing vector
field of constant length on $M$. Then the one-parameter isometry
group $\mu(t)$, $t \in \mathbb{R}$, of the space $M$, generated by
the field $X$, consists of Clifford-Wolf translations. Moreover, if
the space $M$ has positive sectional curvature, then the flow
$\mu(t)$, $t\in \mathbb{R}$, admits a factorization up to a free
isometric action of the circle $S^1$ on $M$.
\end{theorem}

From Theorems \ref{rcwhom} and \ref{Posit}, we get the following

\begin{theorem}\label{scwh}
If a  restrictively Clifford-Wolf homogeneous Riemannian manifold
$(M,\mu)$ is symmetric space, then it is strongly Clifford-Wolf
homogeneous.
\end{theorem}

\begin{proof}
Let $x,y$ are arbitrary points in $(M,\mu)$. Take a shortest
geodesic $\gamma$ in $(M,g),$ joining points $x$ and $y$. Then by
Theorem \ref{rcwhom}, the geodesic $\gamma$ parameterized by the
arclength is an integral curve of a unit Killing vector field $X$ on
$(M,g)$. Now, by Theorem \ref{Posit}, the one-parameter isometry
group $\mu(t)$, $t \in \mathbb{R}$, of the space $M$, generated by
the field $X$, consists of Clifford-Wolf translations. It is clear
that for sufficiently small $|t|$, $\gamma(t)$ shifts all points of
$(X,d)$ to distance $|t|$, and $\gamma(s)(x)=y$, where $d(x,y)=s$.
\end{proof}

\section{The proof of the main result}

We shall need the following useful proposition.

\begin{pred}[Proposition 2.35 in \cite{Bes1}]\label{rcw1}
If the Levi-Civita derivative of the curvature tensor $R$ of a
Riemannian manifold $(M,g)$ satisfies the condition $(\nabla _X
R)(\cdot,X)X=0$ for any $X\in TM$, then $(M,g)$ is locally
symmetric.
\end{pred}

Now, using Theorem \ref{symm1} we can prove the following

\begin{theorem}\label{rcw2}
Every restrictively Clifford-Wolf homogeneous Riemannian manifold
$(M,g)$ is locally symmetric.
\end{theorem}

\begin{proof}
According to Proposition \ref{rcw1}, in order to prove that $(M,g)$
is locally symmetric it suffices to prove $(\nabla_X R)(\cdot,X)X=0$
for any $X\in M_x$ at a fixed point $x\in M$. By Theorem
\ref{rcwhom} any geodesic $\gamma$ on $(M,g)$ is an integral curve
of a Killing vector field of constant length on $(M,g)$, hence, we
can find a Killing field of constant length $Z$ on $(M,g)$ such that
$Z(x)=X$. By Theorem \ref{symm1} we get $(\nabla_Z R)(\cdot,Z)Z=0$
at every point of $M$. In particular,
$$
(\nabla_X R)(\cdot,X)X=(\nabla_{Z(x)} R)(\cdot,Z(x))Z(x)=0.
$$
The theorem is proved.
\end{proof}

\begin{pred}\label{ozon}
Let $M$ be a simply connected compact irreducible symmetric space,
that is not isometric to a Lie group with a bi-invariant Riemannian
metric. If $M$ admits nontrivial Killing vector of constant length
$X$, then either $M$ is an odd-dimensional sphere $S^{2n-1}$ for
$n\geq 3$, or $M=SU(2n)/Sp(n)$, $n\geq 3$.
\end{pred}

\begin{proof}
In fact the assertion of the theorem follows from results of the
paper \cite{Wolf62}, where it is proved that among irreducible
compact simply connected symmetric spaces only odd-dimensional
spheres, spaces $SU(2m)/Sp(m)$, $m\geq 2$, and simple compact Lie
groups, supplied with some bi-invariant Riemannian metrics, admit
one-parameter groups of Clifford-Wolf translations. Here we give a
more direct proof.

Let $G$ be the identity component of the full isometry group of $M$.
Consider a one-parameter isometry group $\mu(t)$, $t\in \mathbb{R}$,
generated by $X$ (in fact, this group consists of Clifford-Wolf
translations of $M$ \cite{BerNikTrans}). By Lemma 1 in
\cite{BerNikTrans} $Z_{\mu}$, the centralizer of the flow $\mu$ in
$G$, acts transitively on $M$ (note that this result is based on
some results of V.~Ozols \cite{Ozols4}). It is clear that the identity
component $K=K(Z_{\mu})$ of $Z_{\mu}$ is a Lie group, which acts
transitively on $M$ and has a non-discrete center (this center
contains $\mu(t)$, $t\in \mathbb{R}$).

Note that $M=G/H$ is a homogeneous space, where $H$ is the isotropy
subgroup at some point $x\in M$. Moreover, according to assumptions
of the theorem, $G$ and $H$ are connected (recall that $M$ is simply
connected) and $G$ is a simple compact Lie group.

In Theorem 4.1 of the paper \cite{Oni1}, A.L.~Onishchik  classified
all connected proper subgroups $K$ of the group $G$, that act
transitively on the homogeneous space $G/H$, where $G$ is a simple
compact connected Lie group and $H$ is its connected closed
subgroup. If in this situation $K$ has a non-discrete center, then
this theorem implies that either $G/H$ is an odd-dimensional sphere,
or $G/H=SU(2n)/Sp(n)$, $n\geq 3$. Since the center of $K(Z_{\mu})$
is non-discrete, it proves the proposition.
\end{proof}

\begin{pred}\label{ozon1}
Let $M$ be a symmetric space $SU(2n)/Sp(n)$, where $n\geq 3$. Then
every Killing vector field of constant length on $M$ has the form
$\Ad(s)(tU)$, where $s\in SU(2n)$, $t \in \mathbb{R}$,
$U=\sqrt{-1}\diag(1,1,\dots,1,-(2n-1))\in su(2n)$. Moreover, $M$ is
not restrictively Clifford-Wolf homogeneous.
\end{pred}

\begin{proof}
According to Theorem 4.1 in \cite{Oni1}, there is a unique (up to
conjugation in $SU(2n)$) connected subgroup with non-discrete center
$K\subset SU(2n)$ that acts transitively on the homogeneous space
$M=SU(2n)/Sp(n)$ ($n\geq 3$). This is the group $SU(2n-1)\times
S^1$, where $SU(2n-1)$ is embedded in $SU(2n)$ via $A\rightarrow
\diag(A,1)$ and $S^1=\exp(tU)$, $t \in \mathbb{R}$,
$U=\sqrt{-1}\diag(1,1,\dots,1,-(2n-1))\in su(2n)$.

If $SU(2n-1)\times S^1$ is the centralizer of a Killing field $V\in
su(2n)$, then clearly $V=tU$ for a suitable $t \in \mathbb{R}$. It
is clear also that any such Killing field has constant length on
$M$, since it lies in the center of the Lie algebra of the group
acting transitively on $M$. This proves the first assertion of the
proposition.

Note that $\dim(SU(2n)/Sp(n))=(n-1)(2n+1)$. On the other hand, we
can easily calculate the dimension of the set of Killing fields of
constant length on $M$. Indeed, this set is $\Ad(SU(2n))(tU)$ (the
orbit of $tU\in su(2n)$ under the adjoin action of the group
$SU(2n)$), $t\in \mathbb{R}$. For any fixed $t$ this orbit is
$SU(2n)/S(U(2n-1)\cdot U(1))$, and $\dim(SU(2n)/S(U(2n-1)\cdot
U(1)))=4n-2$. Since $(4n-2)+1 < (n-1)(2n+1)=\dim(M)$ for $n\geq 3$,
then $M$ is not restrictively Clifford-Wolf homogeneous. Otherwise,
by Theorem \ref{rcwhom} for any $x\in M$ and $U\in M_x$ there is a
Killing vector field of constant length $X$ on $M$ such that
$X(x)=U$, which is impossible by previous calculation of dimensions.
\end{proof}

Now we can prove the main result of the paper.

\begin{proof}[Proof of Theorem \ref{main}]
Let $M$ be a simply connected Clifford-Wolf homogeneous Riemannian
manifold. By Lemma \ref{first} $M$ is complete.
Therefore, by Theorem \ref{rcw2} $M$ is symmetric space. Let us
consider the de~Rham decomposition
$$
M=M_{0}\times M_{1}\times \dots \times M_{k},
$$
where $M_0$ is Euclidean space, and the others $M_i$, $1\leq i \leq
k$, are simply connected compact irreducible symmetric spaces. By
Theorem \ref{product2}, every $M_i$, $0 \leq i \leq k$, is
Clifford-Wolf homogeneous.

From Propositions \ref{ozon} and \ref{ozon1} we get that every
simply connected compact irreducible symmetric space that is
Clifford-Wolf homogeneous is either an odd-dimensional sphere of
constant curvature, or a simply connected compact simple Lie group
supplied with a bi-invariant Riemannian metric. This proves the necessity.

On the other hand, a direct metric
product of an Euclidean space, odd-dimensional spheres of constant
curvature and simply connected compact simple Lie groups supplied
with bi-invariant Riemannian metrics (all these manifolds are Clifford-Wolf homogeneous)
is a Clifford-Wolf homogeneous Riemannian manifold.
\end{proof}

\section{Clifford-Killing spaces}

The following proposition is evident.

\begin{pred}\label{CKvs}
A collection $\{X_1, \dots, X_l\}$ of Killing vector fields on a
Riemannian manifold $(M,g)$ constitutes a basis of a
finite-dimensional vector space $CK_l$ (over $\mathbb{R}$) of
Killing vector fields of constant length if and only if vector
fields $X_1, \dots, X_l$ are linearly independent and all inner
products $g(X_i,X_j)$; $i,j=1,\dots, l$ are constant. Under this
$CK_l$ admits an orthonormal basis of (unit) Killing vector fields.
\end{pred}

We shall call such a space $CK_l$, $l\geq 1$, a
\textit{Clifford-Killing space} or simply \textit{CK-space}. Below
we give a simple method to check the condition $g(X,Y)=\const$ for
given Killing vector fields $X,Y$.

\begin{lemma}\label{nabla}
Suppose  $X$ and $Y$ are Killing fields on a Riemannian manifold
$(M,g)$. Then a point $x\in M$ is a critical point of the function
$x\mapsto g_x(X,Y)$ if and only if
$$
\nabla_X Y =-\nabla_Y X=\frac{1}{2}[X,Y]
$$
at this point.
\end{lemma}

\begin{proof}
Since $\nabla_X Y -\nabla_Y X=[X,Y]$, it suffices to prove that $x$
is a critical point of $g(X,Y)$ if and only if $\nabla_X Y +\nabla_Y
X=0$ at the point $x$. For any Killing vector field $W$ and
arbitrary vector fields $U$ and $V$ on $(M,g)$ we have the equality
$g(\nabla_UW,V)+g(U,\nabla_V W)=0$. Since $X$ and $Y$ are Killing
vector fields, for any vector field $Z$ we get
$$
0=Z \cdot g(X,Y)=g(\nabla_Z X,Y)+g(X,\nabla_Z Y)=
$$
$$
-g(Z,\nabla_YX)-g(\nabla_X Y,Z)=-g(Z,\nabla_YX+\nabla_X Y),
$$
that proves the lemma.
\end{proof}

\begin{corollary}\label{nabla1}
Suppose  $X$ and $Y$ are Killing vector fields on a Riemannian
manifold $(M,g)$. Then $g(X,Y)=\const$ if and only if
$$
\nabla_X Y =-\nabla_Y X=\frac{1}{2}[X,Y].
$$
In particular, a Killing vector field $X$ has constant length if and
only if $\nabla_X X=0$.
\end{corollary}

\begin{definition}
Let $V$ and $W$ be some CK-spaces on $(M,g)$. We say that $V$ is
(properly) equivalent to $W$ if there exists a (preserving
orientation) isometry $f$ of $(M,g)$ onto itself such that
$df(V)=W.$
\end{definition}

The following question is quite interesting.

\begin{question}\label{quest}
Classify all homogeneous Riemannian manifolds which admit nontrivial
CK-spaces. For any such manifold, classify up to (proper)
equivalence all (in particular all maximal by inclusion) possible
CK-spaces.
\end{question}

This difficult question mainly has not been considered before. All
consequent sections are related to this question. We shall see
that it is closely connected with some impressive classical and
recent results.

\section{Riemannian manifolds with the Killing property}

In the paper \cite{DatN}, J.E.~D'Atri and H.K.~Nickerson studied
Riemannian manifolds with {\it the Killing property}.

\begin{definition}[\cite{DatN}]\label{kilpr}
A Riemannian manifold $(M,g)$ is said to have the Killing property
if, in some neighborhood of each point of $M$, there exists an
orthonormal frame $\{X_1,\dots,X_n\}$ such that each $X_i$, $i =
1,\dots , n$, is a Killing vector field (local infinitesimal
isometry). Such a frame is called a Killing frame.
\end{definition}

It is easy to see that every Lie group supplied with a bi-invariant
Riemannian metric has the (global) Killing property. Note that a
generalization of the Killing property is {\it the divergence
property}, that we shall not treat in this paper (see details in
\cite{DatN}).

\begin{pred}\label{ksym}
Every Riemannian manifold with the Killing property is locally
symmetric.
\end{pred}

\begin{proof}
By definition, for any point $x$ in a given manifold $(M,g)$ there
is a Killing frame $\{X_1,\dots,X_n\}$ in some neighborhood $U$ of
the point $x$. Since $g(X_i,X_j)=\delta_{ij}$, then for any real
constants $a_i$, the local vector field $a_1X_1+a_2X_2+\cdots + a_n
X_n$ is a local Killing field of constant length. As a corollary,
for any vector $v\in M_x$ there is a Killing field $Z$ of constant
length in $U$ such that $Z(x)=v.$ Then, by the proof of Theorem
\ref{symm1}, which really doesn't require a global character of
vector fields, we get that $(\nabla_Z R)(\cdot, Z)Z(x)=0$ or
$(\nabla_v R)(\cdot, v)v=0$. By Proposition \ref{rcw1}, $(M,g)$ is
locally symmetric.
\end{proof}

\begin{remark}
Another proof of the previous proposition is given in \cite{DatN}.
\end{remark}

It is well known that every (not necessarily complete) locally
symmetric Riemannian manifold is locally isometric to a symmetric
space (see e.g. \cite{Hel}). Therefore, local properties of
manifolds with the Killing property could be obtained from the study
of complete (simply connected) Riemanian manifolds.

\begin{pred}\label{kcw}
Every simply connected complete Riemannian manifold $(M,g)$ with the
Killing property is symmetric and strongly Clifford-Wolf
homogeneous.
\end{pred}

\begin{proof}
Since $(M,g)$ is locally symmetric (Proposition \ref{ksym}) complete
and simply connected, then it is a symmetric space (see e.g.
\cite{Hel}). In particular, $(M,g)$ is really analytic. For any
point $x$ in a given manifold $(M,g)$ there is a Killing frame
$\{X_1,\dots,X_n\}$ in some neighborhood $U$ of the point $x$. We
state now that the above (local) Killing frame is uniquely extends
to a global Killing frame in $(M,g)$ (we will save the same notation
for corresponding global frame). In fact, every locally defined
Killing vector field on $(M,g)$ is a restriction of a globally
defined Killing vector field (for instance, by Theorems 1 and 2 in
\cite{N}, each local Killing vector field in any simply connected
really analytic Riemannian manifold  admits a unique extension to a
Killing vector field on the whole manifold).

Since $(M,g)$ is really analytic, all functions $g(X_i,X_j)$, must
be constant, and so $g(X_i,X_j)=\delta_{ij}$. Then any linear
combination of vector fields $X_1,\dots,X_n$ over $\mathbb{R}$ is a
Killing vector field of constant length. By completeness and Theorem
\ref{rcwhom}, we get that $(M,g)$ is restrictively Clifford-Wolf
homogeneous. Now it is enough to apply Theorem \ref{scwh}.
\end{proof}

The authors of \cite{DatN} tried but failed to classify simply
connected complete Riemannian manifolds with the Killing property
(see the page 407 right before the section 5 in \cite{DatN}). The
following theorem solves this question completely.

\begin{theorem}\label{kpcw}
A simply connected complete Riemannian manifold $(M,g)$ has the
Killing property if and only if it is isometric to a direct metric
product of Euclidean space, compact simply connected simple Lie
groups with bi-invariant metrics, and round spheres $S^7$ (some
mentioned factors could be absent).
\end{theorem}

\begin{proof}
The sufficiency of this statement follows from the well-known fact
that any mentioned factor has the Killing property (see also the
next section as far as round $S^7$ is concerned).

Let us prove the necessity. As a corollary of Theorems \ref{scwh}
and \ref{main}, $(M,g)$ must have a form as in Theorem \ref{main}.
But we may leave only $S^7$ as factors among odd-dimensional spheres
by the following reason. By Theorem 4.1 in \cite{DatN}, every factor
of the corresponding product also has the Killing property. As is
well known (see Corollary \ref{1,3,7} below), among (simply
connected) odd-dimensional round spheres, only $S^3$ and $S^7$ have
the Killing property. But round $S^3$ can be considered as compact
simply connected simple Lie group $SU(2)$ supplied with a
bi-invariant Riemannian metric. So we can omit also $S^3$.
\end{proof}

\section{Results of A.~Hurwitz and J.~Radon}\label{RH}

A.~Hurwitz has posed the following question.

\begin{question}
\label{vopros} For a given natural number $m$, to find the maximal
natural number $p=\rho(m)$ such that there is a bilinear real
vector-function $z=(z_1,\dots, z_m)=z(x,y)$ of real vectors
$x=(x_1,\dots, x_p)$ and $y=(y_1,\dots, y_m)$, which satisfies the
equation
\begin{equation}\label{Hurw}
\left( x_1^2+\dots + x_p^2 \right) \left(y_1^2+\dots + y_m^2\right)=
z_1^2+\dots + z_m^2.
\end{equation}
for all $x\in \mathbb{R}^p,$ $y\in \mathbb{R}^m.$
\end{question}

Using some equivalent formulations, J.~Radon in \cite{Rad} and
A.~Hurwitz in \cite{H} independently obtained the following answer.

\begin{theorem}\label{Rad}
If $m=2^{4\alpha+\beta}m'$, where $\beta=0,1,2,3$, $\alpha$ is a
non-negative integer, and $m'$ is odd, then $\rho(m)=8\alpha +
2^{\beta}$.
\end{theorem}

Let us give some equivalent formulations and corollaries, following
in some respect to J.~Radon. It is clear that any bilinear function
$z=z(x,y)$ can be represented in the form
\begin{equation}
\label{matrix}
z=z(x,y)=\sum_{j=1}^{p}x_j(A_jy)=\left(\sum_{j=1}^{p}x_jA_j\right)y,
\end{equation}
where $A_j$ are $(m\times m)$-matrices and one consider $z$ and $y$
as vectors-columns. Putting $x_i=\delta_{ji}$, $i=1,\dots, p$ for a
fixed $j$ from $\{1,\dots, p\}$ and using equation (\ref{Hurw}), one
can easily see that

1) every $A_j$ must be an orthogonal matrix.

On the other side, the same equation shows that

2) for any fixed $y\in \mathbb{R}^m,$ $A_jy$, $j=1,\dots, p$ must be mutually
orthogonal vectors in $\mathbb{R}^m$ with lengths, all equal to $|y|$.

As a corollary, we must have $p\leq m$. At the end, the equation
(\ref{Hurw}) and the last form of equation (\ref{matrix}) show that

3) for any unit vector $x\in \mathbb{R}^p$, the matrix
$\left(\sum_{j=1}^{p}x_jA_j\right)$ must be orthogonal.

The last statement explains the title of the paper \cite{Rad} and
gives an equivalent form of the question \ref{vopros}, which
considered J.~Radon.

Now, it is clear that if we change every matrix $A_j$ in the
bilinear form (\ref{matrix}) by $B_j=A_jA$, where $A$ is a fixed
orthogonal $(m\times m)$-matrix, then we get another bilinear form,
which also satisfies the equation (\ref{Hurw}). If we take
$A=A_p^{-1}$, and denote $B_j$ one more by $A_j$, we get the
following form of (\ref{matrix}):
\begin{equation}
\label{matrix1}
z=z(x,y)=\sum_{j=1}^{p-1}x_j(A_jy)+x_p(I y)=\left(\sum_{j=1}^{p-1}x_jA_j+x_pI\right)y.
\end{equation}
Now, applying the properties 1) and 2) for new bilinear form
(\ref{matrix1}), we get that all matrices $A_j$, $j=1,\dots p-1$, in
(\ref{matrix1}) must be both orthogonal and skew-symmetric (since
$A_jy\perp y=Iy$ for all $y\in \mathbb{R}^m$) (it implies in
particular that $p\geq 2$ is possible only when $m$ is even). Now
from Theorem \ref{Rad}, 1), 2) and the last statement it easily
follows

\begin{theorem}
\label{max} Any nontrivial collection of vector fields on $S^{m-1}$
consists of mutually orthogonal unit Killing vector field on
$S^{m-1}$ if and only if it can be presented in a form
$X_j(y)=A_jy$, $y\in S^{m-1}$, $j=1,\dots, p-1$, where $(m\times
m)$-matrices $A_j$, which are both orthogonal and skew-symmetric
(thus $m$ is even), are taken from the equation (\ref{matrix1}),
defining a bilinear form satisfying the equation (\ref{Hurw}). The
maximal number of such fields is equal to $\rho(m)-1$, see Theorem
\ref{Rad}.
\end{theorem}

Theorem  \ref{max} implies

\begin{theorem}\label{ro}
A maximal dimension $l$ of Clifford-Killing spaces $CK_{l}$ on
$S^{m-1}$ is equal to $\rho(m)-1$.
\end{theorem}

\begin{corollary}\label{1,3,7}
A maximal dimension $l$ of Clifford-Killing spaces $CK_{l}$ on
$S^{m-1}$ is equal to $m-1>0$ if and only if $m\in \{2,4,8\}$. As a
corollary, $S^1$, $S^3,$ and $S^7$ are all round spheres with the
Killing property.
\end{corollary}

Note that the last result is related to the existence of algebras of
complex, quaternion, and the Caley numbers.

One should note also that J.F.~Adams has proved that there is no
$\rho(m)$ continuous orthonormal (or, equivalently, linear
independent) tangent vector fields on the sphere $S^{m-1}$
\cite{Adams} (see also Theorem 13.10 of Chapter 15 in \cite{Hus}).

Later on, Eckmann reproved the Hurwitz-Radon Theorem in \cite{E}.
The methods of Radon and Hurwitz yield complicated schemes for
actually constructing the forms (\ref{matrix1}) for $p=\rho(m)$,
which have been simplified by Adams, Lax, and Phillips \cite{ALP},
as well as by Zvengrowski \cite{Z} and by Balabaev \cite{Bal}.

\section{Clifford-Killing spaces on spheres and Clifford algebras and modules}\label{struc}

In this section we continue the study of Clifford-Killing spaces on
spheres $S^{m-1}$, supplied with canonical metrics $\can$ of
constant sectional curvature $1$. As it has been noted, these
spheres are strongly Clifford-Wolf homogeneous, if $m$ is even.

\begin{theorem}\label{CW}
Real $(m\times m)$-matrices $u_{1},u_{2},\dots,u_{l}$, $l\geq 1$,
define pairwise orthogonal at every point of $S^{m-1}$ unit Killing
vector fields $U_{i}(x):=u_{i}x, x\in S^{m-1}$, on $S^{m-1}$ if and
only if

\begin{equation}
\label{onet} u_{i}\in O(m)\cap so(m), \quad u_{i}^{2}=-I,\quad
i=1,\dots,l,
\end{equation}
and
\begin{equation}
\label{ant}
u_{i}u_{j}+u_{j}u_{i}=0,\quad i\neq j.
\end{equation}
In this case $m$ is even and $u_{i}\in SO(m)$.
\end{theorem}

\begin{proof}
The vector fields $U_i$ are unit Killing vector fields on $S^{m-1}$
if and only if they can be presented as follows: $U_{i}(x)=u_{i}x,
\quad x\in S^{m-1},$ where $u_i\in O(m)\cap so(m)$. This implies
that $m$ is even, $(x,y)=(u_ix,u_iy)=-(u_i^2x,y),$ and $u_i^2=-I$ (here $(\cdot, \cdot)$
means the standard inner product in $\mathbb{R}^m$).

It is clear that under the first condition in (\ref{onet}), vector
fields $U_i$ and $U_j$ on $S^{m-1}$ are orthogonal if and only if
$$((u_i+u_j)x, (u_i+u_j)x)=(u_ix,u_ix)+(u_jx,u_jx), \quad x\in S^{m-1},$$
which is equivalent to the identity $(u_iu_jx,x)\equiv 0$,
$x \in \mathbb{R}^m$, or $u_iu_j\in so(m)$, or
$$-(u_iu_j)=(u_iu_j)^t=u_j^tu_i^t=(-u_j)(-u_i)=u_ju_i,$$ i.~e.
(\ref{ant}).

A matrix $u_{i}\in O(m)$ is skew symmetric if and only if $u_{i}$ is
orthogonally similar to a matrix $u=\diag(C,\dots,C)$, where $C\in
O(2)$ is skew symmetric. This implies that $u_{i}\in SO(m)$.
\end{proof}

\begin{remark}
According to Theorem \ref{CW}, later on we shall suppose that $m$
is even and $m=2n$.
\end{remark}

Theorem \ref{CW} naturally leads to the notion of associative
algebras over the field $\mathbb{R}$, with generators
$e_{1},\dots,e_{l},$ such that $e_{i}^{2}=-1$,
$e_{i}e_{j}+e_{j}e_{i}=0$, $i\neq j$ (and any other relation in the
algebra is some corollary of the indicated relations). Such algebra
$Cl_{l}$ is called \textit{the Clifford algebra} (with respect to
negatively definite quadratic form $-(y,y)$ in $\mathbb{R}^k$ and
orthonormal basis $\{e_1,\dots ,e_l\}$ in $(\mathbb{R}^k,(\cdot,
\cdot))$). These algebras include the algebra $Cl_1=\mathbb{C}$ of
complex numbers and the algebra $Cl_2=\mathbb{H}$ of quaternions.

The others Clifford algebras can be described as follows:
$Cl_3=\mathbb{H}\oplus \mathbb{H}$,
$Cl_4=\mathbb{H}\otimes_{\mathbb{R}} \mathbb{R}(2)$, where the
algebra $\mathbb{R}(2)$ is generated by symmetric $(2\times
2)$-matrices $\diag\{1,-1\}$ and the permutation matrix of vectors
in canonical basis $\{e_1,e_2\}$ for $\mathbb{R}^2$. After this one
can apply "the periodicity law"
$Cl_{k+4}=Cl_k\otimes_{\mathbb{R}}Cl_4$. See more details in
\cite{Hus, GM}.

Let us consider $L_{l,2n}$, the algebra of linear operators (on
$\mathbb{R}^{2n}$) that is generated by the operators (the matrices)
$u_{i}$, $i=1,\dots,l\geq 1$. Theorem \ref{CW} implies that
$L_{l,2n}$ is a homomorphic image of the Clifford algebra $Cl_{l}$.
It is easy to see that the kernel of the natural homomorphism $c:
Cl_{l}\rightarrow L_{l,2n}$ is a two-sided ideal of $Cl_l$. The
above description of Clifford algebras shows that $Cl_l$, $l\neq
4k+3$, contains no proper two-sided ideal, while the Clifford
algebra $Cl_{l}=Cl_{l-1}\oplus Cl_{l-1}$ contains exactly two proper
two-sided ideals $A_1$ and $A_2$ that are both isomorphic to
$Cl_{l-1}$ if $l=4k+3$. Thus, in any case $L_{l,2n}$ is isomorphic
either to $Cl_{l}$, or to $Cl_{l-1}$ (if $l\neq 4k+3$, then
necessarily $L_{l,2n}$ is isomorphic to $Cl_{l}$).

Now we consider an example, where $L_{l,2n}$ is not isomorphic to
$Cl_{l}$.

\begin{example}
Let us consider a Lie algebra $CK_{3}\subset so(2n)$, where $2n=4k$.
In this case $CK_{3}$ is the linear span of vectors
$u_{1},u_{2},u_{3}:=u_{1}u_{2}$. Clearly, the algebra $L_{3,2n}$ is
not isomorphic to $Cl_{3}$. It is easy to calculate the dimensions
of both these associative algebras. The dimension of $Cl_{m}$
considered as a vector space over $\mathbb{R}$, is $2^{m}$
\cite{Hus}. In particular, $\dim (Cl_{3})=8$. On the other hand,
$L_{3,2n}=\Lin \{1,u_{1},u_{2},u_{3}\}$, $\dim (L_{3,2n})=4$,
because
$$
u_{1}u_{3}=u_{1}u_{1}u_{2}=-u_{2},\quad
u_{2}u_{3}=u_{2}u_{1}u_{2}=u_{1}.
$$
Notice that in this case $L_{3,2n}$ is isomorphic to the quaternion
algebra $\mathbb{H}=Cl_{2}$.
\end{example}

The above mentioned homomorphism $c: Cl_l\rightarrow L_{l,2n}$
defines a \textit{representation of Clifford algebra} $Cl_l$ in
$\mathbb{R}^{2n}$, so the last vector space together with the
representation $c$ becomes a \textit{Clifford module} (over $Cl_l$).
We saw that any Clifford-Killing space $CK_l\subset so(2n)$ defines
the structure of Clifford module on $\mathbb{R}^{2n}$ over $Cl_l$.
It is very important that in this way one can get \textit{arbitrary}
Clifford module.

We need some information on the classification of Clifford modules
over $Cl_l$ \cite{Hus}:

a) If $l\neq 4k+3,$ then there exists (up to equivalence) precisely
one irreducible Clifford module $\mu_l$ over $Cl_l$ with the
representation $c_l$. Every Clifford module $\nu_l$ over $Cl_l$ is
isomorphic to the $m$-fold direct sum of $\mu_l,$ that is,
\begin{equation}\label{dec}
\nu_l \cong \oplus^m \mu_l.
\end{equation}

b) If $l=4k+3$, then there exist (up to equivalence) precisely two
non-equivalent irreducible Clifford modules $\mu_{l,1}, \mu_{l,2}$
over $Cl_l$ with representations $c_1=c_{l-1}\circ \pi_1$ and
$c_2=c_{l-1}\circ \pi_2,$ where $\pi_i$ is natural projection of
$Cl_l$ onto the ideal $A_i$, $i=1,2$. The modules $\mu_1, \mu_2$
have the same dimension and every Clifford module $\nu_l$ over
$Cl_l$ is isomorphic to
\begin{equation}\label{dec1}
\nu_l \cong \oplus^{m_1}\mu_{l,1} \oplus \oplus^{m_2}\mu_{l,2},
\end{equation}
for some non-negative integers $m_1, m_2$. It is clear that the
representation of $Cl_l$ corresponding to the module $\nu_l$ is
exact if and only if both numbers $m_1$ and $m_2$ are non-zero.

The dimension $n_0$ of $\mu,$ or $\mu_1, \mu_2$ is equal to
$n_0=2^{4\alpha+\beta}$, if $8\alpha+2^{\beta-1}-1< l\leq
8\alpha+2^{\beta}-1$, where $\alpha$ is a non-negative integer and
$\beta=0,1,2,3$. In some sense, this is the function, inverse to
\textit{the Radon-Hurwitz function} $\rho(m)$ from Theorem
\ref{Rad}.

Previous discussion implies

\begin{theorem}\label{uni}
The sphere $S^{2n-1}$ admits a Clifford-Killing space $CK_l$ if and
only if $1\leq l\leq \rho(2n)-1$. In this case $n_0=n_0(l)$ divides
$2n$. All Clifford-Killing spaces $CK_{l}$ for $S^{2n-1}$ are
pairwise equivalent (all spaces $CK_{l}\subset so(2n)$ are
equivalent with respect to $O(2n)$) if and only if $l\neq 4k+3$. If
$l=4k+3$, then there exist exactly $[n/n_0(l)]+1$ non-equivalent
classes of Clifford-Killing spaces $CK_{l}$ for $S^{2n-1}$. In
particular, all Clifford-Killing spaces $CK_{\rho(2n)-1}$ for
$S^{2n-1}$ are pairwise equivalent if and only if
$2n=2^{4\alpha+\beta}n'$, where $\alpha$ is a non-negative integer,
$n'$ is odd, and $\beta=1$ or $\beta=3$. If $\beta=0$ or $\beta=2$
above, then there exist exactly $[n'/2]+1$ non-equivalent classes of
Clifford-Killing spaces $CK_{\rho(2n)-1}$ for $S^{2n-1}$.
\end{theorem}

This theorem together with Theorem \ref{end} below give an exact
number of (proper) equivalence classes for Clifford-Killing spaces
on $S^{2n-1}$.

\begin{pred}\label{cw1}
The set of unit Killing vector fields on the sphere $S^{2n-1}$
represents by itself a union of two disjoint orbits with respect to
the adjoint action of the group $SO(2n)$, and one orbit with respect
to the adjoint action of the group $O(2n)$.
\end{pred}

\begin{proof}
By Theorem \ref{CW}, an arbitrary unit Killing vector field on the
sphere $S^{2n-1}$ is defined by a matrix $U$ from $SO(2n)\cap
so(2n)$ with the condition $U^2=-I$. Thus there is a matrix $A(U)\in
O(2n)$ such that $A(U)UA(U)^{-1}=\diag(C,\dots,C)$, where $C$ is a
fixed matrix $C\in SO(2)\cap so(2)$. Moreover, if $A'(U)$ is another
such matrix, then $A(U)[A'(U)]^{-1}\in SO(2n)$. This implies that
every two matrices $U, V\in SO(2n)\cap so(2n)$ are equivalent in
$O(2n)$, and equivalent in $SO(2n)$ if and only if $A(U)A(V)^{-1}\in
SO(2n).$ This finishes the proof.
\end{proof}

\begin{theorem}\label{end}
If $2n\equiv 2 ({\rm mod}\, 4)$, then any two spaces of the type
$CK_l$ (necessarily $l=1$) are $SO(2n)$-equivalent. If $2n\equiv
0({\rm mod}\, 4)$, then any $O(2n)$-equivalence class of spaces $CK_l
\subset so(2n)$ contains exactly two $SO(2n)$-equivalence classes.
\end{theorem}

\begin{proof}
We shall use the vectors $U$ and $V$ as in Proposition \ref{cw1},
where $A(U)\in SO(2n)$ and $A(V)\in O(2n)\setminus SO(2n)$. Let
suppose that $2n\equiv 2({\rm mod}\, 4)$. Then $l=1$ and any space
$CK_1$ is spanned onto some unit Killing vector field $X$. Note that
$X$ is equivalent with $-X\in CK_1$ by means an orthogonal matrix
with determinant $-1$, thus either $X$ or $-X$ is equivalent to the
vector $U$ in $SO(2n)$.

Let suppose now that $2n\equiv 0({\rm mod}\, 4)$. Consider any space
$CK_l$, which is spanned onto unit Killing vector fields
$U_1,\dots,U_l$. If $l=1$, then $X_1$ is obviously equivalent to
$-X_1$ in $SO(2n)$. If $l > 1$, then any of two unit Killing vector
fields in $CK_l$ can be continuously deformed (in the set $CK_l$) to
another one. So all unit Killing vector fields in $CK_l$ are
simultaneously equivalent to (only one of) $U$ or $V$ in $SO(2n)$,
for example, to $U$.

Let's consider now $B\in O(2n)\setminus SO(2n)$ and the space
$CK^{\prime}_l$, which is equivalent to $CK_l$ by $B$. Since nor
$U$, neither $-U$ is equivalent to $V$ in $SO(2n)$, then any unit
Killing vector in $CK^{\prime}_l$ is equivalent to the vector $V$,
which is not equivalent in $SO(2n)$ to unit Killing vectors in
$CK_l$. Thus spaces $CK_l$ and $CK^{\prime}_l$ are not equivalent to
each other in $SO(2n)$. On the other hand, since $SO(2n)$ is an
index 2 subgroup of $O(2n)$, it is clear that any
$O(2n)$-equivalence class contains at most two $SO(2n)$-equivalence
classes.
\end{proof}

It should be noted that subspaces $CK_l$ of Lie algebras $so(2n)$
play an important role in various mathematical theories. For
example, Theorem \ref{uni} and statements on the page 23 in
\cite{BTV} imply that there is a bijection between $O(2n)$-classes
of $CK_l\subset so(2n)$ (for all possible pairs $(2n,k)$ of this
type) and isometry classes of {\it generalized Heisenberg groups}
studied at first by A.~Kaplan in \cite{Kap}. These are special
two-step nilpotent groups admitting a one dimensional solvable
Einstein extensions that are well-known {\it Damek-Ricci spaces}
\cite{DR}. Note that Damek-Ricci spaces are harmonic Riemannian
manifolds and most of them are not symmetric. We refer the reader to
\cite{BTV, EH, GorKer,Barb, CD, Heb2} and references therein for a
deep theory of generalized Heisenberg groups and Damek-Ricci spaces.
Note also that there are some useful generalizations of subspaces of
the type $CK_l$ in Lie algebras $so(2n)$. One of them is a notion of
{\it uniform subspaces} of $so(2n)$ \cite{GorKer}. Such subspaces
are used for producing new Einstein solvmanifolds with two-step
nilpotent nilradical (see \cite{GorKer} and \cite{Kerr3} for
details).

\section{Clifford-Killing spaces for $S^{2n-1}$ and Radon's unit spheres in $O(2n)$}

Now we supply the Lie algebra $so(2n)$ with the following
$\Ad(SO(2n))$-invariant inner product:
\begin{equation}\label{GK}
(U,V)=-\frac{1}{2n}\trace(UV).
\end{equation}
The Lie group $SO(2n)$ supplied with the corresponding bi-invariant
inner Riemannian metric $\rho$, is a symmetric space. This metric is
uniquely extended to bi-invariant ``metric'' $\rho$ on $O(2n)$
($\rho(x,y)=+\infty$ if and only if $x,y$ lie in different connected
components).

Note that $(X,X)=1$ for every unit Killing field $X$ on the sphere
$S^{2n-1}$ supplied with the canonical metric of constant curvature
$1$, and the Killing form $B$ of $so(2n)$ is connected with the form
(\ref{GK}) by the formula
$$
B(U,V)=2(n-1) \trace (UV)=-4n(n-1)(U,V).
$$
Let us remind that forms $\trace(UV)$ and $B(U,V)$ on $so(2n)$ are
forms \textit{associated} with the identical and the adjoint representations
of $so(2n)$ respectively \cite{Dix}.

\begin{definition}\label{rads}
Let $A_1, \dots, A_p, 1\leq p\leq \rho(2n)$ are matrices from
$O(2n)$, defining a bilinear form (\ref{matrix}). Then the set of
all matrices of the form $\sum_{i=1}^{p}x_iA_i,$ where
$x=(x_1,\dots, x_p)$ is a unit vector in $\mathbb{R}^p$, we will
call \textit{Radon's unit sphere for the form (\ref{matrix})} or
simply \textit{Radon's unit sphere}, and denote by $RS^{p-1|2n}$.
\end{definition}

It follows from Section \ref{RH} and Theorem \ref{CW} that always
$RS^{p-1|2n}\in O(2n)$, and $RS^{p-1|2n}\in SO(2n)$ (respectively
$RS^{p-1|2n}\notin SO(2n)$) if and only if at least one of matrices
$A_1, \dots, A_p$ is in $SO(2n)$ (is not in $SO(2n)$). Moreover,
$RS^{p-1|2n}A$ is also Radon's unit sphere for every $A\in O(2n)$,
and left and right translations are isometries in $(SO(2n),\rho)$.
Thus, from geometric viewpoint, we can consider only the Radon unit
sphere, defined by a form (\ref{matrix1}). As it was said in Section
\ref{RH}, in this case $A_p=I,$ and the linear span of $A_1, \dots,
A_{p-1}$ is a Clifford-Killing space $CK_l\subset so(2n)$, where
$l=p-1.$ The goal of this section is the following

\begin{theorem}\label{Radsph}
In the notation of Definition \ref{rads}, the map
$$x=(x_1,\dots, x_{l+1=p})\in S^{l}\subset \mathbb{R}^{l+1}\rightarrow
\sum_{j=1}^{p}x_j(A_j)\in (O(2n),\rho)$$ preserves distances and has
the image $RS^{l|2n}$. As a corollary, any Radon's sphere
$RS^{l|2n}$, defined by a form (\ref{matrix}), is a totally geodesic
submanifold in $(O(2n),\rho)$, which is isometric to the standard
unit sphere $S^l\subset \mathbb{R}^{l+1}$ with unit sectional
curvature. If $A_{p}=I$, it is equal to $\exp(CK_l)$, where $CK_l$
is the linear span of $A_1, \dots, A_{l=p-1}$ from the formula
(\ref{matrix1}). Conversely, the image $\exp(CK_l)\subset SO(2n)$ of
arbitrary Clifford-Killing space $CK_l\subset so(2n)$ is a Radon's
unit sphere $RS^{l|2n}$ for some form (\ref{matrix1}).
\end{theorem}

\begin{proof}
At first we shall prove the third assertion. It was proved in
Section \ref{RH} and Theorem \ref{CW} that all $A_1, \dots,
A_{p-1=l}$ are elements of $SO(2n)\cap so(2n),$ which defines
mutually orthogonal unit Killing vector fields on $S^{2n-1}$. Let
$C\in RS^{l|2n}\subset SO(2n)$, that is
$C=\sum_{i=1}^{l}x_iA_i+x_{l+1}I$, where $x=(x_1,\dots ,x_{l+1})$ is
a unit vector in $\mathbb{R}^{l+1}$. Let suppose at first that
$x_{l+1}^2\neq 1$. Then the matrix
$$
A:=\frac{1}{\sqrt{\sum_{i=1}^{l}x_i^2}}\sum_{i=1}^{l}x_iA_i
$$
is in $SO(2n)\cap so(2n)\cap CK_l$ and defines a unit Killing vector
field on $S^{2n-1}$. Obviously the vector $C$ can be represented in
a form $C=(\cos r)I+(\sin r)A, r\in \mathbb{R}$, where $\cos
r=x_{l+1}$. Now for any $t, s \in \mathbb{R}$ we have $A^2=-I$ by
Theorem \ref{CW} and so
$$
[(\cos t)I+(\sin t)A][(\cos s)I+(\sin s)A]=(\cos (t+s))I+(\sin (t+s))A.
$$
This means that the set of matrices $(\cos t)I+(\sin t)A$, $t\in
\mathbb{R}$,  constitutes a one-parameter subgroup in $SO(2n)$ with
the tangent vector $A$, and $C=\exp(rA)\in \exp(CK_l)$. If
$x_{l+1}=1$ or $x_{l+1}=-1$, then for any matrix $A\in CK_l$
defining a unit Killing vector field, we have by the above arguments
that $\exp(tA)\in RS^{l|2n}$ for all $t\in \mathbb{R}$, and
$\exp(0A)=C$ or $\exp(\pi A)=C$. So, we have proved the required
equality $RS^{l|2n}=\exp(CK_l)$. This implies in particular that
$\exp(CK_l)$ is topologically $S^l$.

By Definition \ref{CKvs}, any subspace $CK_{l=p-1}$ on $S^{2n-1}$
has a basis of $l$ mutually orthogonal unit Killing vector fields,
which are defined by some matrices $A_1,\dots A_{p-1}$ in
$SO(2n)\cap so(2n)$ by Theorem \ref{CW}.(Let us note also that by
the same Theorem, $A_iA_j=-A_jA_i$, if $i\neq j$, and so
$(A_i,A_j)=0$ by the formula (\ref{GK}).) It is clear now that the
formula (\ref{matrix1}) defines a required bilinear form, and the
last assertion follows from the third one.

By the the last two assertions, $RS^{l|2n}=\exp{CK_l}$, where $CK_l$
is the linear span of $A_1,\dots, A_{l}$, where $l=p-1$. Now, if
$A\in CK_l\cap SO(2n)$, then by Theorem \ref{CW}, $A^2=-I$ and the
formula (\ref{GK}) gives us that $(A,A)=1$. This fact, the
$\Ad(SO(2n))$-invariance of the inner product (\ref{GK}), and the
discussions in the first part of the proof implies that $\exp(tA)$,
$t\in \mathbb{R}$, is a geodesic circle in $SO(2n)$ of the length
$2\pi$, entirely lying in $RS^{l|2n}$.

Let suppose that $B$ and $C$ are two different points in
$RS^{l|2n}$, defined by unit vectors $b$ and $c$ in
$\mathbb{R}^{l+1}$. Then there is a unit vector $d\in
\mathbb{R}^{l+1}$, which is orthogonal to $b$, such that $c=(\cos
r)b+(\sin r)d$ for some $r\in [0,\pi]$. The vector $d$ defines the
corresponding element $D\in RS^{l|2n}$. Matrices $A_1'=B,
A_{2=p}'=D$ define a bilinear form $z(x,y)$ with necessary
properties by the formula (\ref{matrix}), and corresponding Radon's
sphere $RS^{1,2n}\subset O(2n),$ containing the points $A,D,C$. By
the previous results, the right translation by $D^{-1}$, which is an
isometry on $(O(2n),\rho)$, transforms this Radon's sphere to
another one of the form $\exp(CK_1)$, which is a geodesic circle in
$(SO(2n),\rho)$ of the length $2\pi$. This implies that the curve
$c(t):=(\cos t)B+(\sin t)D, t\in [O,r]$, joining the points $B$ and
$C$ in $RS^{k|2n}$, is a shortest geodesic in $(O(2n),\rho)$,
parameterized by the arc-length. Thus we have proved the second
assertion. The first assertion follows from the second one and the
bi-invariance of the ``metric'' $\rho$ on $O(2n)$.
\end{proof}

\section{Lie triple systems in $so(2n)$ and totally geodesic spheres in $SO(2n)$}

Recall that a linear subspace $\mathfrak{a}$ of a Lie algebra
$\mathfrak{g}$ is called {\it Lie triple system} if
$[\mathfrak{a},[\mathfrak{a},\mathfrak{a}]]\subset \mathfrak{a}$. We
shall need the following

\begin{lemma}\label{triple0}
If $\mathfrak{a}$ is a Lie triple system of a Lie algebra
$\mathfrak{g}$, then $\mathfrak{h}:=[\mathfrak{a}, \mathfrak{a}]$
and $\mathfrak{k}:=\mathfrak{h}+\mathfrak{a}$ are subalgebras of
$\mathfrak{g}$, $\mathfrak{h}\cap \mathfrak{a}$ is an ideal of
$\mathfrak{k}$.
\end{lemma}

\begin{proof} From $[\mathfrak{a},[\mathfrak{a},\mathfrak{a}]]\subset \mathfrak{a}$ and
the Jacobi identity we get $[\mathfrak{h},\mathfrak{h}]\subset
\mathfrak{h}$ and $[\mathfrak{h},\mathfrak{a}]\subset \mathfrak{a}$,
that proves the first assertion of the lemma. Now, it is easy to see
that $[\mathfrak{a},\mathfrak{h}\cap \mathfrak{a}]\subset
\mathfrak{h}\cap \mathfrak{a}$ and $[\mathfrak{h},\mathfrak{h}\cap
\mathfrak{a}]\subset \mathfrak{h}\cap \mathfrak{a}$, which proves
the second assertion.
\end{proof}

\begin{lemma}\label{triple}
Let $G$ be a compact Lie group supplied with a bi-invariant
Riemannian metric $\rho$. Suppose that $\mathfrak{a}$ is a Lie
triple system in $\mathfrak{g}$, the Lie algebra of $G$. Then
$M:=\exp(\mathfrak{a})$ is a totally geodesic submanifold of $G$, in
particular, $M$ is a symmetric space. If universal Riemannian
covering of $M$ is irreducible, then one of the following conditions
holds:

1) $\mathfrak{a}$ is a simple Lie subalgebra of $\mathfrak{g}$ and
$M$ is a simple compact Lie group with a bi-invariant Riemannian
metric.

2) The Lie algebra $\mathfrak{h}:=[\mathfrak{a}, \mathfrak{a}]$ satisfies the relation
$\mathfrak{h}\cap \mathfrak{a}=0$, and $(\mathfrak{h}\oplus
\mathfrak{a}, \mathfrak{h})$ is a symmetric pair corresponding to $M$.
In particular, $\exp(\mathfrak{h}\oplus
\mathfrak{a})$ is the full connected isometry group of $M$.
If in this case the Lie algebra
$\mathfrak{h}\oplus \mathfrak{a}$ is not simple, then
$M$ is a simple compact Lie group with a bi-invariant Riemannian
metric.
\end{lemma}

\begin{proof}
Here we consider $\exp(\mathfrak{a})$
and  $\exp(\mathfrak{h}\oplus \mathfrak{a})$  in
the Lie theoretical sense. On the other hand,
since $\rho$ is bi-invariant, they could be treated also in
the the Riemannian sense. Let $(\cdot, \cdot)$ be a
$\Ad(G)$-invariant inner product on $\mathfrak{g}$ that generates
$\rho$.

Let us consider the standard representation of $(G,\rho)$ as a
symmetric space: $G\times G /\diag(G)$. Consider
$\widetilde{\mathfrak{g}}=\mathfrak{g}\oplus \mathfrak{g}$, the Lie
algebra of $G\times G$, and
$$
\widetilde{\mathfrak{k}}=\diag(\mathfrak{g})\subset
\widetilde{\mathfrak{g}}, \quad
\widetilde{\mathfrak{p}}=\{(X,-X)\,|\, X\in \mathfrak{g}\}.
$$
Then $\widetilde{\mathfrak{g}}=\widetilde{\mathfrak{k}}\oplus
\widetilde{\mathfrak{p}}$ is a Cartan decomposition for the
symmetric space $G\times G /\diag(G)$. Now consider
$\widetilde{\mathfrak{a}}=\{(X,-X)\,|\, X\in \mathfrak{a}\}\subset
\widetilde{\mathfrak{p}}$. It is clear that
$\widetilde{\mathfrak{a}}$ is a Lie triple system in
$\widetilde{\mathfrak{g}}$,
$[\widetilde{\mathfrak{a}},\widetilde{\mathfrak{a}}]=\mathfrak{h}\oplus
\mathfrak{h} \subset \widetilde{\mathfrak{k}}$. Using the standard
theory of Lie triple system in symmetric spaces (see e.g.
\cite{Hel}), we conclude that the pair
$([\widetilde{\mathfrak{a}},\widetilde{\mathfrak{a}}]\oplus
\widetilde{\mathfrak{a}}, \widetilde{\mathfrak{a}})$ is symmetric
and corresponds to the symmetric space
$M_1:=\exp(\widetilde{\mathfrak{a}})$. It is clear that
$M_1$ (supplied with Riemannian metric induced by the
inner product $\frac{1}{2}(\cdot,\cdot)+\frac{1}{2}(\cdot,\cdot)$ on
$\mathfrak{g}\oplus \mathfrak{g}$) is isometric to $M$.
In this case it is sufficient to consider an isometry $i:M\rightarrow M_1$ defined as follows:
$i(\exp(tX))=(\exp(tX),\exp(-tX))$, where $t\in \mathbb{R}$, $X\in \mathfrak{a}$.
This proves
the first assertion of the lemma.

Let us suppose that the universal covering of $M=M_1$ is
an irreducible symmetric space. Then the Lie algebra
$[\widetilde{\mathfrak{a}},\widetilde{\mathfrak{a}}]\oplus
\widetilde{\mathfrak{a}}$, being the Lie algebra of the full
isometry group of $M_1$, is either simple or a direct sum
of two copies of some simple Lie algebra. Recall that
$\mathfrak{u}:= \mathfrak{h}\cap \mathfrak{a}$ is an ideal in the
Lie algebra $\mathfrak{h}+\mathfrak{a}$ by Lemma \ref{triple0}. Let
us consider maps
$$\pi_1 :[\widetilde{\mathfrak{a}},\widetilde{\mathfrak{a}}]\oplus \widetilde{\mathfrak{a}} \rightarrow
\mathfrak{h} + \mathfrak{a}, \quad \pi_2
:[\widetilde{\mathfrak{a}},\widetilde{\mathfrak{a}}]\oplus
\widetilde{\mathfrak{a}} \rightarrow \mathfrak{h} +
\mathfrak{a},
$$
defined as follows. Let $Z=(Y,Y)+(X,-X)$, where $X\in \mathfrak{a}$
and $Y\in [{\mathfrak{a}},{\mathfrak{a}}]$, then we put $\pi_1
(Z)=Y+X$ and $\pi_2 (Z)=Y-X$. It is easy to see that $\pi_1$ and
$\pi_2$ are Lie algebra epimorphisms. The kernel of $\pi_1$ is
$\mathfrak{u}_1=\{(0,X)\,|\, X\in \mathfrak{u}\}$ and the kernel of
$\pi_2$ is $\mathfrak{u}_2=\{(X,0)\,|\, X\in \mathfrak{u}\}$. In
particular, we get that
$\mathfrak{u}\oplus \mathfrak{u}=\mathfrak{u}_1\oplus \mathfrak{u}_2$ is an ideal in the
Lie algebra
$[\widetilde{\mathfrak{a}},\widetilde{\mathfrak{a}}]\oplus
\widetilde{\mathfrak{a}}$.

Suppose that the Lie algebra
$[\widetilde{\mathfrak{a}},\widetilde{\mathfrak{a}}]\oplus
\widetilde{\mathfrak{a}}$ is isomorphic to $\mathfrak{s}\oplus \mathfrak{s}$, where $\mathfrak{s}$ is a simple
Lie algebra. If $\mathfrak{u}$ is not trivial, then
$[\widetilde{\mathfrak{a}},\widetilde{\mathfrak{a}}]\oplus
\widetilde{\mathfrak{a}}$
coincides with its ideal $\mathfrak{u}\oplus
\mathfrak{u}=\mathfrak{u}_1\oplus \mathfrak{u}_2$ (this ideal could not be proper in this case).
Moreover, $\mathfrak{u}=\mathfrak{a}=[\mathfrak{a},\mathfrak{a}]=\mathfrak{h}$ is a simple
Lie algebra (since $\mathfrak{u}_1$ is the kernel of ${\pi}_1$, then $\mathfrak{h} + \mathfrak{a}$
is isomorphic to $\mathfrak{u}_2\sim \mathfrak{u}=\mathfrak{h}\cap \mathfrak{a}$).
In this case $M$ is a connected simple compact Lie
group supplied with a bi-invariant Riemannian metric, and we get Condition 1) of the lemma.
If $\mathfrak{u}$ is trivial, then $\pi_1$ is an isomorphism.
Therefore, $\mathfrak{h}+ \mathfrak{a}=\mathfrak{h}\oplus \mathfrak{a}$ is isomorphic to
$\mathfrak{s}\oplus \mathfrak{s}$,
and $\mathfrak{h}$ is isomorphic to $\diag (\mathfrak{s})$ (since $\diag (\mathfrak{s})$ is a
unique proper subalgebra in $\mathfrak{s}\oplus \mathfrak{s}$).
Hence $M$ is a simple compact Lie group (with the Lie algebra isomorphic to $\mathfrak{s}$)
with a bi-invariant Riemannian
metric, and we get Condition 2) of the lemma with non-simple $\mathfrak{h}\oplus \mathfrak{a}$.

Now, if the Lie algebra
$[\widetilde{\mathfrak{a}},\widetilde{\mathfrak{a}}]\oplus
\widetilde{\mathfrak{a}}$ is simple, then $\mathfrak{u}$ is trivial and $\pi_1$ is an isomorphism.
Therefore, $\mathfrak{h}+ \mathfrak{a}=\mathfrak{h}\oplus \mathfrak{a}=
\pi_1([\widetilde{\mathfrak{a}},\widetilde{\mathfrak{a}}]\oplus
\widetilde{\mathfrak{a}})$ is a simple Lie algebra that is the Lie
algebra of the full isometry group of $M$. Moreover, by the definition of $\pi_1$ we obtain
$\pi_1([\widetilde{\mathfrak{a}},\widetilde{\mathfrak{a}}])=\mathfrak{h}$. Therefore,
Condition 2) of the lemma with simple $\mathfrak{h}\oplus \mathfrak{a}$ holds.
\end{proof}

\begin{theorem}\label{cw3}
Let $CK_l$ be a Clifford-Killing subspace of dimension $l$ in the
Lie algebra $so(2n)$. Then the following statements are true:

1) $CK_l$ is a Lie triple system in the Lie algebra $so(2n)$.

2) For every $CK_1$ the image $\exp(CK_1)$ is a closed geodesic of
length $2\pi$ in $(SO(2n),\rho)$. If $l\geq 2$, then the image
$\exp(CK_l)$ is a totally geodesic sphere $S^l$ of constant
sectional curvature $1$ in  $(SO(2n),\rho)$.

3) If $CK_l$ is  a Lie subalgebra of $so(2n)$ then either $l=1$, or
$l=3$.

4) Every $CK_1$ is a commutative Lie subalgebra of $so(2n)$, the
image $\exp(CK_1)$ consist of Clifford-Wolf translations of
$S^{2n-1}$, and the $\exp(CK_1)$-orbits constitute a totally
geodesic foliation of equidistant (great) circles  in the sphere
$S^{2n-1}$.

5) If $CK_3$ is a Lie subalgebra of $so(2n)$ then $CK_3=[CK_3,CK_3]$
is isomorphic to $so(3)\sim su(2)$, and $\exp(CK_3)=S^3$ is the
group $SU(2)$ supplied with a bi-invariant Riemannian metric.
Moreover, $\exp(CK_3)=SU(2)$ consists of Clifford-Wolf translations
of $S^{2n-1}$ and, consequently, the $\exp(CK_3)$-orbits constitute
a totally geodesic foliation of equidistant (great) $3$-spheres  in
the sphere $S^{2n-1}$.

6) If $CK_l$ is not a Lie subalgebra of $so(2n)$ then $CK_l \cap
[CK_l,CK_l]$ is trivial, subspaces $CK_l \oplus [CK_l,CK_l]$ and
$[CK_l,CK_l]$ are Lie subalgebras of $so(2n)$ such that the pair
$(CK_l \oplus [CK_l,CK_l],[CK_l,CK_l])$ is the symmetric pair
$(so(l+1), so(l))$ and $\exp (CK_l+[CK_l,CK_l])$ is isomorphic to
$SO(l+1)$, the full connected isometry group of $\exp(CK_l)=S^l$.
\end{theorem}

\begin{remark}
Some variants of this theorem are admissible in the literature (see
e.g. \cite{Eber}).
\end{remark}

\begin{proof}
Let $CK_l$ be the linear span of  unit Killing vector fields
$U_1,\dots, U_l$. By Theorem \ref{CW} we get that $U_i^2=-I$ and
$U_iU_j+U_jU_i=0$ for $i\neq j$. Let us show that
$[CK_l,[CK_l,CK_l]]\subset CK_l$. If the indices $i,j,k$ are
pairwise distinct, then $[U_i,U_j]=2U_iU_j$ and
$$
[U_k,[U_i,U_j]]=2U_kU_iU_j-2U_iU_jU_l= -2U_iU_kU_j+2U_iU_kU_j=0.
$$
On the other hand,
$$
[U_i,[U_i,U_j]]=2U_iU_iU_j-2U_iU_jU_i= -2U_j+2U_iU_iU_j=-4U_j \in
CK_l.
$$
Therefore, $CK_l$ is a Lie triple system in $so(2n)$, that proves
the first assertion of the theorem.

The second assertion immediately follows from Theorem \ref{Radsph}.

If $CK_l$ is a Lie algebra, then necessarily $l=1$ or $l=3$, that
proves the third assertion of the theorem.

The fourth assertion immediately follows from Theorem \ref{Posit}.

If $CK_3$ is a Lie subalgebra of $so(2n)$, then $CK_3$ is isomorphic
to $so(3)\sim su(2)$. In this case $\exp(CK_3)=SU(2)=S^3$ with a
metric of constant sectional curvature $1$. Obviously, every $X\in
CK_3$ is a Killing vector field of constant length on $S^{2n-1}.$ By
Theorem \ref{Posit}, we get that $\exp(CK_3)=SU(2)$ consists of
Clifford-Wolf translations of $S^{2n-1}$. Now, it easy to see that
$\exp(CK_3)$-orbits constitute a totally geodesic foliation of
equidistant (great) $3$-spheres  in the sphere $S^{2n-1}$. This
proves the fifth assertion of the theorem.

By Lemma \ref{triple}, $CK_l +[CK_l,CK_l]$ is a Lie subalgebra of
$so(2n)$ and the subgroup $\exp (CK_l+[CK_l,CK_l])$ of $SO(2n)$ is a
connected isometry group of $M:=\exp(CK_l)=S^{l}$. If $l=3$ and
$CK_3$ is not a Lie subalgebra of $so(2n)$, then (by Lemma
\ref{triple}) the Lie algebra $CK_3+[CK_3,CK_3]=CK_3\oplus
[CK_3,CK_3]$ is isomorphic to $so(4)\sim so(3)\oplus so(3)$ and
$\exp(CK_3 \oplus [CK_3,CK_3])=SO(4)$.  Now, suppose that $l\neq 1,
3$. Then the sphere $S^l=\exp(CK_l)$ is not a Lie subgroup of
$SO(2n)$. By Lemma \ref{triple}, we get that $CK_l \cap [CK_l,CK_l]$
is trivial, the Lie algebra $CK_l \oplus [CK_l,CK_l]$ is $so(l+1)$
and $\exp(CK_l \oplus [CK_l,CK_l])=SO(l+1)$.

The theorem is completely proved.
\end{proof}

According to Theorems \ref{Radsph} and \ref{cw3}, the study of
subspaces $CK_l$ is related to the study of totally geodesic spheres
in $SO(2n)$.

Note that there are well known totally geodesic spheres in $SO(2n)$,
{\it the Helgason spheres}. In the paper \cite{HelP}, S.~Helgason
proved that every compact irreducible Riemannian symmetric space $M$
with the maximal sectional curvature $\varkappa$ contains totally
geodesic submanifolds of constant curvature $\varkappa$. Any two
such submanifolds of the same dimensions are equivalent under the
full connected isometry group of $M$. The maximal dimension of any
such submanifold is $1+m(\overline{\delta})$, where
$m(\overline{\delta})$ is the multiplicity of the highest restricted
root $\overline{\delta}$. Moreover, if $M$ is not a real projective
space, then such submanifolds of dimension $1+m(\overline{\delta})$
are actually spheres.

In the case when $M$ is a simple compact Lie group with a
bi-invariant Riemannian metric, $m(\overline{\delta})=2$. Therefore,
the maximal dimension of submanifolds as above is $3$. Note that
$M=(SO(2n),\rho)$, $n\geq 2$,  is not a real projective space, therefore, there
are $3$-dimensional totally geodesic Helgason's spheres in
$(SO(2n),\rho)$. It is easy to give a description of these spheres
(see \cite{HelP}). At first we recall some well known facts (see
e.g. $H$ of Chapter $8$ in \cite{Bes}). Let $E_{ij}$ be a $(2n\times
2n)$-matrix with zero entries except the $(i,j)$-th entry that is
$1$. Consider matrices $F_i=E_{(2i)(2i-1)}-E_{(2i-1)(2i)}$ for
$1\leq i\leq n$. These matrices define a basis of a standard Cartan
subalgebra $\mathfrak{k}$ in $so(2n)$. Hence every $X$ in
$\mathfrak{k}$ has the form $X=\sum\limits_{i=1}^n {\lambda}_i F_i$.
Every root (with respect to $\mathfrak{k}$) of $so(2n)$ has the form
$\lambda_i \pm \lambda_j$, $i\neq j$. Note that all these roots have
the same length.

Let $V_{\lambda_i \pm \lambda_j}$ be the (two-dimensional) root
space of the root $\lambda_i \pm \lambda_j$ in $so(2n)$. Put
$U_{\lambda_i \pm \lambda_j}=\mathbb{R}\cdot (F_i \pm F_j) \oplus
V_{\lambda_i \pm \lambda_j}$. In this notation, $\exp(U_{\lambda_i \pm
\lambda_j})$ is a Helgason's sphere in $(SO(2n),\rho)$
(see details in the proof of Theorem 1.2 in \cite{HelP}).
Moreover, by Theorem 1.1 in \cite{HelP} any two Helgason's spheres in $(SO(2n),\rho)$
are equivalent under the
full connected isometry group of $(SO(2n),\rho)$. Therefore, every Helgason's sphere in $(SO(2n),\rho)$
is conjugate in $SO(2n)$ either to the sphere $\exp(U_{\lambda_1 - \lambda_2})$, or
to the sphere $\exp(U_{\lambda_1 + \lambda_2})$.

\begin{pred}\label{helcl1} The spheres $S^3=\exp(CK_3)$ in
Assertion 5 of Theorem \ref{cw3} are Helgason's spheres of constant curvature $1$ in $(SO(4),\rho)$. Every
Clifford-Killing space $CK_3$ in $so(4)$ is an ideal of $so(4)$.
\end{pred}

\begin{proof}
Note that $so(4)\cong so(3)\oplus so(3)$. There are only two
non-proportional roots for the standard Cartan algebra
$\mathfrak{k}$: $\lambda_1 + \lambda_2$ and $\lambda_1 - \lambda_2$.
It is easy to see that $U_{\lambda_1 + \lambda_2}$ and $U_{\lambda_1 - \lambda_2}$ are pairwise
commuting Lie algebras isomorphic to $so(3)\sim su(2)$.  In particular, $\exp(U_{\lambda_1 \pm \lambda_2})$
are spheres of the type $\exp(CK_3)$ in $(SO(4), \rho)$
(see Assertion 5 in Theorem \ref{cw3}).
Note also, that $U_{\lambda_1 + \lambda_2}$ and $U_{\lambda_1 - \lambda_2}$
could be naturally identified with the Lie algebras of left and
right shifts on $S^3=SU(2)$.
On the other hand,
(from the previous discussion) $\exp(U_{\lambda_1 \pm \lambda_2})$
are Helgason's spheres $(SO(4), \rho)$. By Assertion 2 in Theorem \ref{cw3} (or by Theorem \ref{Radsph})
these spheres have constant curvature $1$. This proves the first assertion of the proposition.

Now, let $CK_3$ be an arbitrary Clifford-Killing space
of $so(4)$. By Theorem \ref{Radsph},
$\exp(CK_3)$ is a sphere $S^3$ of constant curvature $1$.
Since its sectional curvature coincides with the sectional curvature of
Helgason's spheres $\exp(U_{\lambda_1 \pm \lambda_2})$, then
$\exp(CK_3)$ should be a Helgason's sphere too. From the description
of Helgason's spheres right before the statement of the proposition we get, that $\exp(CK_3)$ is
conjugate in $SO(4)$ either to $\exp(U_{\lambda_1 - \lambda_2})$, or to $\exp(U_{\lambda_1 + \lambda_2})$.
In particular, $CK_3$
is conjugate in $SO(4)$ either to $U_{\lambda_1 - \lambda_2}$, or to $U_{\lambda_1 + \lambda_2}$.
Since $U_{\lambda_1 \pm \lambda_2}$ are ideals in $so(4)$, this proves the second assertion of the proposition.
\end{proof}

\begin{pred}\label{helcl2}
Every Helgason's sphere in $(SO(2n),\rho)$, $n\geq 2$, has the
constant sectional curvature $n/2$. In particular, for $n\geq 3$ all
Helgason's spheres are distinct from the spheres in Theorem
\ref{Radsph}.
\end{pred}

\begin{proof}
For $n=2$  the assertion of the proposition follows from Proposition
\ref{helcl1}. Now, let us consider the case $n
\geq 3$. The subgroup $H=\diag(SO(4),1,\dots, 1)\subset SO(2n)$ with
a Riemannian bi-invariant metric $\rho_1$ induced by $\rho$, is a
totally geodesic submanifold in $(SO(2n),\rho)$. On the other hand,
all roots of $\mathfrak{h}={\rm Lie}(H)$ are roots of $so(2n)$. From the
above description of the Helgason spheres we get that every
Helgason's sphere in $(H,\rho_1)$ is also a Helgason's sphere in
$(SO(2n),\rho)$. It is easy to see that $(H, \frac{n}{2} \rho_1)$ is
isometric to $(SO(4),\rho^{\prime})$, where $\rho^{\prime}$ is a bi-invariant Riemannian metric, generated
by the inner product (\ref{GK}) for $n=2$. Since
all Helgason's spheres in $(SO(4),\rho^{\prime})$ has constant curvature $1$ by Proposition
\ref{helcl1},
then every Helgason's sphere in $(H,\rho_1)$ has constant curvature
$n/2$.
\end{proof}

\section{Lie algebras in Clifford-Killing spaces on $S^{2n-1}$}

Below we discuss some results related to Lie algebras containing in
Clifford-Killing subspaces $CK_l$ of $so(2n)$.

\begin{pred}\label{Lie}
Let $X,Y$ be linearly independent Killing vector fields on
$S^{2n-1}$ with constant inner products $\can(X,X)$, $\can(Y,Y)$,
$\can(X,Y)$. Then the Lie bracket $[X,Y]$ is a non-trivial Killing
vector field of constant length on $S^{n-1}$. If $X,Y$ are unit
mutually orthogonal Killing vector fields on $S^{2n-1}$, then the
triple of vector fields
$$\{X,Y,Z:=\frac{1}{2}[X,Y]\}$$
constitutes an orthonormal basis of Lie algebra $CK_{3}$ of vector
fields on $S^{2n-1}$ with relations $[X,Y]=2Z$, $[Z,X]=2Y$,
$[Y,Z]=2X$. Consequently, Assertion 5) of Theorem \ref{cw3} holds.
\end{pred}

\begin{proof}
By Corollary \ref{nabla1} we get $\frac{1}{2}[X,Y]=\nabla_X Y
=-\nabla_Y X$, and by Proposition \ref{criv2}, the formula
$\can(\nabla_X Y, \nabla_X Y)= \can(R(X,Y)Y,X)$ holds. Since
$\can(X,Y)$ and the sectional curvature of $(S^{n-1},\can)$ are
constant, the expression $\can(R(X,Y)Y,X)$ is a positive constant.
Consequently, $[X,Y]$ is a (non-trivial) Killing field of constant
length, that proves the first assertion of the proposition.

The orthonormality of the triple of Killing vector fields
$\{X,Y,Z\}$ on $S^{2n-1}$ follows from the orthonormality of the
pair $\{X,Y\},$ the proof of the first part of Proposition, and
relations $0=Y\can(X,X)=-\can([X,Y],X)$, $0=X\can(Y,Y)=\can([X,Y],Y)$.
The first commutation relation follows from the definition. On the
ground of Theorem \ref{CW}, we get the equality:
$$[Z,X]=\frac{1}{2}\{(XY-YX)X-X(XY-YX)\}=$$
$$\frac{1}{2}\{XYX-YXX-XXY+XYX\}=\frac{1}{2}(2Y+2Y)=2Y.
$$
The third commutation relation can be proved analogously. Then the
linear span $CK_{3}$ of vectors $X,Y,Z$ is a Lie algebra. Now, it
suffices to apply Theorem \ref{cw3}.
\end{proof}

\begin{pred}
\label{no} If $l\geq 4$, then the space $CK_{l}$ contains no Lie
subalgebra of dimension $\geq 2$.
\end{pred}

\begin{proof} This proposition immediately follows from Theorem
\ref{cw3}. Indeed, in this case the intersection $CK_l \cap
[CK_l,CK_l]$ is trivial, while $CK_l$ contains no two-dimensional
commutative Lie subalgebra, because $S^{l}$ is a CROSS.
\end{proof}

\begin{corollary}
A sphere $S^{2n-1}$ admits the space $CK_{\rho(2n)-1}$, which is a
Lie algebra if and only if $\rho(2n)=2$ or $\rho(2n)=4$, i. e. when
$8$ doesn't divide $2n$.
\end{corollary}

\begin{proof}
If $8$ divide $2n$, then by Theorem \ref{ro}, the dimension of the
space $CK_{\rho(2n)-1}$ is more than $4$, and on the ground of
Proposition \ref{no}, this space cannot be a Lie algebra.

If $4$ divide $2n$, but $8$ doesn't divide $2n$, then $\rho(2n)-1=3$
and by Proposition \ref{Lie}, as the space $CK_{3}$ one can take the
Lie algebra with the basis $X,Y,Z,$ indicated there.

At the end, if $4$ doesn't divide $2n$, then $\rho(2n)-1=1,$ and any
space $CK_{1}$ is a Lie algebra.
\end{proof}

Propositions \ref{Lie}, \ref{no},  and Theorem \ref{ro} immediately
imply

\begin{corollary}
\label{11} If $8$ divide $2n$, then there is a space $CK_{3}$ that
is a Lie algebra (isomorphic to $su(2)$) which is not contained in
any space $CK_{\rho(2n)-1}$.
\end{corollary}

Propositions  \ref{Lie} and \ref{no} imply

\begin{corollary}
\label{11n} 1) If $4$ divide $2n$, then there is a space
$CK_{3}\subset so(2n)$, which is a Lie algebra, isomorphic to
$so(3).$

2) If $8$ divide $2n$, then there exists a space $CK_{3}\subset
so(2n)$, which is not a Lie algebra.
\end{corollary}

\begin{example}
\label{example} Here we consider two examples of the spaces $CK_3$, where the first
(respectively, second) one is (respectively, is not) a subalgebra of $so(2n)$.
In the Lie algebra $so(4)$, consider
the vectors
$$
U_1=
\begin{pmatrix}
0 & 0 & 0 & -1 \\
0 & 0 & -1 & 0 \\
0 & 1 & 0 & 0 \\
1 & 0 & 0 & 0 \\
\end{pmatrix},
\quad U_2=
\begin{pmatrix}
0 & 0 & -1 & 0 \\
0 & 0 & 0 & 1 \\
1 & 0 & 0 & 0 \\
0 & -1 & 0 & 0 \\
\end{pmatrix},
\quad U_3=
\begin{pmatrix}
0 & -1 & 0 & 0 \\
1 & 0 & 0 & 0 \\
0 & 0 & 0 & -1 \\
0 & 0 & 1 & 0 \\
\end{pmatrix}.
$$
It is easy to check that $U_1^2=U_2^2=U_3^2=-I$, $U_1U_2=-U_3$,
$U_2U_1=U_3$, $U_1U_3=U_2$, $U_3U_1=-U_2$, $U_2U_3=-U_1$,
$U_3U_2=U_1$. Therefore, the linear span of the vectors $U_i$,
$1\leq i\leq 3$, in $so(4)$ is a Lie subalgebra of the type $CK_3$.

Now, consider in $so(8)$ the vectors
$$
\widetilde{U}_1=\diag(U_1,-U_1), \quad
\widetilde{U}_2=\diag(U_2,-U_2), \quad
\widetilde{U}_3=\diag(U_3,-U_3).
$$
It is easy to check that
$\widetilde{U}_1^2=\widetilde{U}_2^2=\widetilde{U}_3^2=-I$,
$$
\widetilde{U}_1\widetilde{U}_2=\diag(-U_3,-U_3), \quad
\widetilde{U}_1\widetilde{U}_3=\diag(U_2,U_2), \quad
\widetilde{U}_2\widetilde{U}_3=\diag(-U_1,-U_1),
$$
$$
\widetilde{U}_2\widetilde{U}_1=\diag(U_3,U_3), \quad
\widetilde{U}_3\widetilde{U}_1=\diag(-U_2,-U_2), \quad
\widetilde{U}_3\widetilde{U}_2=\diag(U_1,U_1).
$$
Therefore, the linear span of the vectors $\widetilde{U}_i$, $1\leq
i\leq 3$, in $so(8)$ is a subspace of the type $CK_3$, but is not a
Lie subalgebra. It is easy to check that the Lie algebra
$CK_3+[CK_3,CK_3]$ is isomorphic to $so(3)\oplus so(3)$.
\end{example}

\medskip
\medskip
\medskip

\end{document}